\documentclass{amsart}
\usepackage{amssymb, amsmath}
\usepackage[latin1]{inputenc}
\usepackage[final]{graphicx}
\usepackage[font=small,labelfont=bf]{caption}
\usepackage{color}

\newtheorem{theorem}{Theorem}[section]
\newtheorem{lemma}[theorem]{Lemma}
\newtheorem{remark}[theorem]{Remark}

\newtheorem{definition}[theorem]{Definition}
\newtheorem{corollary}[theorem]{Corollary}

\makeatletter
 \@addtoreset{equation}{section}
\makeatother

\begin{document}

\title[Solutions to the affine quasi-Einstein equation]
{Solutions to the affine quasi-Einstein equation for homogeneous surfaces}
\author{M. Brozos-V\'{a}zquez \, E. Garc\'{i}a-R\'{i}o\, P. Gilkey \, X. Valle-Regueiro}
\address{MBV: Universidade da Coru\~na, Differential Geometry and its Applications Research Group, Escola Polit\'ecnica Superior, 15403 Ferrol,  Spain}
\email{miguel.brozos.vazquez@udc.gal}
\address{EGR-XVR: Faculty of Mathematics,
University of Santiago de Compostela,
15782 Santiago de Compostela, Spain}
\email{eduardo.garcia.rio@usc.es, javier.valle@usc.es}
\address{PBG: Mathematics Department, University of Oregon, Eugene OR 97403-1222, USA}
\email{gilkey@uoregon.edu}
\thanks{Supported by projects ED431F 2017/03, and MTM2016-75897-P (Spain).}

\begin{abstract}
We examine the space of solutions to the affine quasi--Einstein equation in the context of homogeneous surfaces.
As these spaces can be used to create gradient Yamabe solitions, conformally Einstein metrics, and warped product Einstein manifolds using
the modified Riemannian extension, we provide very explicit descriptions of these solution spaces. We use the dimension of the space of affine
Killing vector fields to structure our discussion as this provides a convenient organizational framework.
\end{abstract}

\maketitle
\section{Introduction}
Let $\mathcal{M}=(M,\nabla)$ be an affine surface. Here $M$ is a smooth connected
surface and $\nabla$ is a torsion free connection on the tangent bundle of $M$. 
The curvature operator and the Ricci tensor are given by
$$
R(X,Y):=[\nabla_X,\nabla_Y]-\nabla_{[X,Y]}\text{ and }\rho(X,Y):=\operatorname{Tr}\{ Z\rightarrow R({Z,X})Y\}\,.
$$
The symmetrization and anti-symmetrization of the Ricci tensor are given by
$$
\rho_{s}(X,Y):=\textstyle\frac12\{\rho(X,Y)+\rho(Y,X)\}\text{ and }\rho_{a}(X,Y):=\textstyle\frac12\{\rho(X,Y)-\rho(Y,X)\}\,.
$$
Let $(x^1,x^2)$ be a system of local coordinates on $M$. To simplify the notation, we let $\partial_{x^i}:=\frac\partial{\partial x^i}$.
Adopt the {\it Einstein convention} and sum over repeated indices. Expand $\nabla_{\partial_{x^i}}\partial_{x^j}=\Gamma_{ij}{}^k\partial_{x^k}$
to define the Christoffel symbols of $\nabla$; since $\nabla$ is torsion free, $\Gamma_{ij}{}^k=\Gamma_{ji}{}^k$. 

\subsection{The affine quasi-Einstein equation}

Let $f\in C^\infty(M)$. Let
$$
\mathcal{H} f=\nabla^2f:=(\partial_{x^i}\partial_{x^j}f-\Gamma_{ij}{}^k\partial_{x^k}f)\,dx^i\otimes dx^j
$$
be the Hessian. The {\it affine quasi-Einstein equation}  is 
\begin{equation}\label{E1.a}
\mathcal{H} f=\mu f \rho_{s}\,,
\end{equation}
where $\mu\in\mathbb{R}$ is the eigenvalue parameter. Denote the associated eigenspace by
$$
E(\mu,\nabla):=\{f\in C^\infty(M):\mathcal{H} f=\mu f \rho_{s}\}\,.
$$
Similarly, if $P\in M$, let $E(P,\mu,\nabla)$ be the germs at $P$ of solutions to Equation~(\ref{E1.a}). We showed \cite{BGGV17a}
that if $M$ is simply connected and if $\dim\{E(P,\mu,\nabla)\}$ is constant on $M$, then any $f\in E(P,\mu,\nabla)$
extends uniquely to an element of $E(\mu,\nabla)$. Thus for the simply connected homogeneous affine geometries, there is no distinction between the
local and the global theory.

In earlier work \cite{BGGV17}, we used the modified Riemannian extension to pass between solutions to Equation~(\ref{E1.a})
and  solutions to the quasi-Einstein equation for manifolds of signature $(2,2)$. 
Let $\pi:T^*M\rightarrow M$ be the canonical projection and let $\phi$ be an
arbitrary symmetric 2-tensor field. If $\vec x=(x^1,x^2)$ is a system of coordinates on $M$, expand a 1-form $\omega=y_1dx^1+y_2dx^2$ to define
coordinates $(x^1,x^2, y_1,y_2)$ on $T^*M$. Construct the modified Riemannian extension on $T^* M$ by setting
\begin{equation}\label{E1.b}
g_{\nabla,\phi}=dx^i\otimes dy_i+dy_i\otimes dx^i+(\phi_{ij}-2y_k\Gamma_{ij}{}^k)dx^i\otimes dx^j\,.
\end{equation}
Let $f\in E(\mu,\nabla)$ be positive. Express $f=e^{-\frac12\mu h}$. Then results of \cite{BGGV17} show that
$(T^*M,g_{\nabla,\phi},\hat h=\pi^*h,\frac12 \mu)$ is a quasi-Einstein isotropic manifold, i.e. we have:
$$
\mathcal{\hat H}\,\hat h+\hat\rho-\textstyle\frac12\,\mu\, d\hat h\otimes d\hat h=\lambda g_{\nabla,\phi}\,\,\text{ for }
\,\,\lambda=0\,\,\text{ and }\,\,\|d\hat h\|=0\,,
$$
where $\mathcal{\hat H}$ and $\hat{\rho}$ are the Hessian and the Ricci tensors on $(T^*M,g_{\nabla,\phi})$.
Furthermore, any self-dual quasi-Einstein isotropic manifold of signature $(2,2)$ with $\mu\neq-1$ has this form \cite{ BGGV17}.

The affine quasi-Einstein equation codifies important geometric information. Solutions corresponding to the eigenvalue
$\mu=0$ give rise to gradient Yamabe solitons on $(T^*M,g_{\nabla,\phi})$ (see \cite{BCGV16}) while solutions for the eigenvalue $\mu=-1$
provide conformally Einstein structures on $(T^*M,g_{\nabla,\phi})$.
Suppose $\mu=\frac{2}{r}$ for $r$ a positive integer. Results of \cite{KimKim} show that if $N$ is a Ricci flat manifold of dimension $r$,
then the warped product $(T^*M,g_{\nabla,\phi})\times_{e^{-\pi^*h/r}}N$ is again an Einstein manifold
if $h$ is a solution to the affine quasi-Einstein equation. Thus it is important to have solutions
to the affine quasi-Einstein equation for quite general eigenvalues $\mu$.

A smooth vector field $X$ on $\mathcal{M}$ is said to be an \emph{affine Killing vector field}
if the Lie derivative $\mathcal{L}_X\nabla$ of $\nabla$ vanishes (see \cite{KN63}). 
Let $\mathfrak{K}(\mathcal{M})$ be the Lie algebra of affine Killing vector fields of $\mathcal{M}$.
In \cite{BGGV17a}, we studied Equation~(\ref{E1.a}) in its own right and in a more general context showing:
\begin{theorem}\label{T1.1}
Let $m=\dim M $. If $f\in E(P,\mu,\nabla)$ satisfies $f(P)=0$ and if $df(P)=0$, then $f$ vanishes identically. 
Consequently, $\dim\{E(P,\mu,\nabla)\}\le m+1$.
Furthermore, if $X$ is the germ of an affine Killing vector field based at $P$, then $XE(P,\mu,\nabla)\subset E(P,\mu,\nabla)$.
Thus $E(P,\mu,\nabla)$ is a module over the Lie algebra of germs of affine Killing vector fields.
\end{theorem}

A connection $\nabla$ is said to be  {\it strongly projectively flat} if $\nabla$
is strongly projectively equivalent to a flat connection, i.e. if there exists a  closed $1$-form $\omega$
on $M$ so that $\Gamma_{ij}{}^k=\omega_i \delta_j^k +\omega_j\delta_i^k$. 
An affine surface $\mathcal{M}$ is strongly projectively flat if and only if both $\rho$ and $\nabla\rho$
are totally symmetric (see, for example \cite{E70, NS, S95}).
The following two results of~\cite{BGGV17b} relate this notion to the dimension of the space of solutions to the affine quasi-Einstein equation.

\begin{theorem}\label{T1.2}
Let  $\mathcal{M}$ be a non-flat  affine surface. Then $\dim\{E(-1,\nabla)\}\ne2$. Moreover  $\dim\{E(-1,\nabla)\}=3$ if and only if $\mathcal{M}$ is strongly projectively flat and
 $\dim\{E(\mu,\nabla)\}=3$ for $\mu\neq-1$ if and only if $\mathcal{M}$ is flat.
\end{theorem}

\begin{theorem}\label{T1.3}
Let  $\mathcal{M}$ be a strongly projectively flat affine surface with Ricci tensor of rank two. Then 
$\dim \{E(0,\nabla)\}=1$ and  $\dim \{E(\mu,\nabla)\}=0$ for $\mu\neq -1 ,0$.
\end{theorem}

\subsection{Aim of the paper}
Theorem~\ref{T1.1} implies that $\dim\{E(\mu,\nabla)\}\le3$ for any affine surface.
In \cite{BGGV17b}, we determined $\dim\{E(\mu,\nabla)\}$ for homogeneous affine surfaces.
One has $E(0,\nabla)=\ker\{\mathcal{H}\}$ is the space of Yamabe solitons; thus $\mu=0$ is a distinguished eigenvalue.
The eigenvalue $\mu=-1$ is distinguished as well. One has, for example, that $\dim\{E(-1,\nabla)\}=3$ is the maximal value
if and only if $\mathcal{M}$ is strongly projectively flat. In this companion paper to \cite{BGGV17b}, instead of simply computing
the dimension of $E(\mu,\nabla)$, we shall determine explicit bases for these spaces since, as noted above, these can be used
to create gradient Yamabe solitons, conformally Einstein metrics and warped product Einstein manifolds by means of modified Riemannian extensions.
In Section~\ref{S3}, we study homogeneous surfaces with $\dim\{\mathfrak{K}(\mathcal{M})\}=4$. In Section~\ref{S4},
we study homogeneous surfaces with $\dim\{\mathfrak{K}(\mathcal{M})\}=3$; these are all of
Type~$\mathcal{B}$ or of Type~$\mathcal{C}$. In Section~\ref{S5},
we conclude our analysis by studying homogeneous surfaces with $\dim\{\mathfrak{K}(\mathcal{M})\}=2$.

\section{Homogeneous surfaces}

A diffeomorphism $\Psi:M_1\rightarrow M_2$ between two  affine surfaces $\mathcal{M}_1=(M_1,\nabla^1)$ 
and $\mathcal{M}_2=(M_2,\nabla^2)$ is said to be an \emph{affine isomorphism} if  $\Psi^*(\nabla^2)=\nabla^1$; in this setting
we have $\Psi^*\{E(\mu,\nabla{}^2)\}=E(\mu,\nabla^1)$ and $\Psi^*\{E(\Psi(P),\mu,\nabla{}^2)\}=E(P,\mu,\nabla^1)$.

An affine surface $\mathcal{M}$ is said to be \emph{homogeneous} if for any two points $P,Q\in\mathcal{M}$,
there exists an affine isomorphism $\Psi:\mathcal{M}\rightarrow\mathcal{M}$ so that $\Psi(P)=Q$; in this setting,
the Lie algebra of affine Killing vector fields $\mathfrak{K}(\mathcal{M})$ is at least $2$-dimensional. 
A $2$-dimensional Lie algebra $\mathfrak{g}=\operatorname{Span}\{ X,Y\}$ is either Abelian 
(i.e., $[X,Y]=0$) or solvable (i.e., $[X,Y]=Y$). Assume that $\mathfrak{K}(\mathcal{M})$ contains 
a  $2$-dimensional Abelian sub-algebra generated by affine Killing vector fields $X$ and $Y$ such 
that $[X,Y]=0$. Then there exist local coordinates $(x^1,x^2)$ so that the corresponding Christoffel 
symbols $\Gamma_{ij}{}^k$ are constant (see Lemma 1 in \cite{AMK08}).
On the other hand, if $\mathfrak{K}(\mathcal{M})$ contains a  $2$-dimensional solvable  sub-algebra
generated by affine Killing vector fields $X$ and $Y$ such that $[X,Y]=Y$,  then there exist local coordinates
$(x^1,x^2)$ so that the corresponding Christoffel symbols satisfy $\Gamma_{ij}{}^k=(x^1)^{-1}C_{ij}{}^k$
for some constants $C_{ij}{}^k$ (see Lemma 2 in \cite{AMK08}).

Work of Opozda \cite{Op04} shows that any homogeneous affine surface $\mathcal{M}$ 
corresponds to one of the above examples and is modeled on one of the following geometries:
\begin{enumerate}
\item[($\mathcal{A}$)] $(\mathbb{R}^2,\nabla)$ where the Christoffel symbols $\Gamma_{ij}{}^k=\Gamma_{ji}{}^k\in\mathbb{R}$ are
constant. $\{\partial_{x^1},\partial_{x^2}\}$ generate an Abelian Lie sub-algebra of Killing vector fields. The translation group 
$(x^1,x^2)\rightarrow(x^1+b^1,x^2+b^2)$ preserves this geometry and
acts transitively on $\mathbb{R}^2$ so this is a homogeneous geometry.
\item[($\mathcal{B}$)] $(\mathbb{R}^+\times\mathbb{R},\nabla)$ with Christoffel symbols $\Gamma_{ij}{}^k=(x^1)^{-1}C_{ij}{}^k$
or $C_{ij}{}^k=C_{ji}{}^k$ constant. $\{x^1\partial_{x^1}+x^2\partial_{x^2},\partial_{x^2}\}$ generate a
non-Abelian Lie sub-algebra of Killing vector fields. The ``$a\vec x+b$" group $(x^1,x^2)\rightarrow(ax^1,ax^2+b)$
preserves this geometry and acts
transitively on $\mathbb{R}^+\times\mathbb{R}$ so this also is a homogeneous geometry.
\item[($\mathcal{C}$)] $(M,\nabla)$ where $\nabla$ is the Levi-Civita connection of a metric of constant sectional curvature.
\end{enumerate}

\begin{remark}\label{R2.1}
\rm These are not exclusive possibilities (see \cite{BGG16}). One has that there are no non-flat surfaces which are both
Type~$\mathcal{A}$ and Type~$\mathcal{C}$. Moreover, a non-flat Type~$\mathcal{B}$ surface is also of 
Type~$\mathcal{A}$ if and only if $(C_{12}{}^1,C_{22}{}^1,C_{22}{}^2)=(0,0,0)$. Finally, a Type~$\mathcal{B}$ surface is also of 
Type~$\mathcal{C}$ if and onlyit is flat or affine isomorphic to a surface whose non-zero Christoffel symbols are
$(C_{11}{}^1,C_{12}{}^2,C_{22}{}^1)=(-1,-1,\pm1)$; the associated Type~$\mathcal{C}$ structure is given by 
$ds^2=(x^1)^{-2}\{(dx^1)^2\pm(dx^2)^2 \}$ and is either the hyperbolic upper half-plane or the Lorentzian analogue.
\end{remark}

The space $\mathfrak{K}(\mathcal{M})$ of affine Killing vector fields plays a role since, by Theorem~\ref{T1.1}, the eigenspaces
$E(\mu,\nabla)$ are modules over this Lie algebra. For an affine surface, $\dim\{\mathfrak{K}(\mathcal{M})\}\leq 6$ (see \cite{KN63}). 
If $\dim\{\mathfrak{K}(\mathcal{M})\}= 6$ then $\mathcal{M}$ is flat. In this case the quasi-Einstein equation becomes $\mathcal{H} f=0$. 
Let $(x^1,x^2)$ be local coordinates so that the Christoffel symbols vanish: $\Gamma_{ij}{}^k=0$. Then  
$E(\mu,\nabla)=E(0,\nabla)=\operatorname{span}\{1,x^1,x^2\}$ and $\dim\{E(\mu,\nabla)\}$ is maximum
 in this instance. Hence we shall assume that $\mathcal{M}$ is not flat henceforth. We shall examine the different possibilities for 
 $\dim\{\mathfrak{K}(\mathcal{M})\}$ seriatim.

We showed previously \cite{BGG16} that $\dim\{\mathfrak{K}(\mathcal{M})\}\in\{2,3,4\}$ if
$\mathcal{M}$ is a non-flat homogeneous surface. Affine surfaces with 
$\dim\{\mathfrak{K}(\mathcal{M})\}=4$ are modeled on Type~$\mathcal{A}$ surfaces and their 
Ricci tensor has rank one (see Theorem 3.11 in \cite{BGG16}). Any Type~$\mathcal{C}$ surface has 
$\dim\{\mathfrak{K}(\mathcal{M})\}=3$. In addition there are exactly other two classes of 
Type~$\mathcal{B}$ geometries with $\dim\{\mathfrak{K}(\mathcal{M})\}=3$  
(see Theorem 3.11 in \cite{BGG16}). The remaining cases correspond to the generic situation where $\dim\{\mathfrak{K}(\mathcal{M})\}=2$.
As we wish to describe the eigenspaces $E(\mu,\nabla)$ very explicitly,
we shall proceed rather combinatorially and rely on previous results \cite{BGG16,BGG16a,BGGV17b}
concerning different models for homogeneous affine surfaces.

\subsection{Linear versus affine equivalence}

Let $\mathcal{M}$ be a Type~$\mathcal{A}$ affine surface model. Any transformation  of the form 
$T(x^1,x^2)=(a^1_1 x^1 + a^1_2 x^2, a^2_1 x^1+a^2_2 x^2)$ where 
$(a_i^j)\in\operatorname{GL}(2,\mathbb{R})$  is an affine  isomorphism. 
We say that two Type~$\mathcal{A}$ surface models are \emph{linearly isomorphic} if
there exists $T\in\operatorname{GL}(2,\mathbb{R})$  intertwining the two structures.
One has (see Lemma 2.1 in \cite{BGG16a})  the following result.

\begin{lemma}\label{L7.1}
Let $\mathcal{M}_1$ and $\mathcal{M}_2$ be two Type~$\mathcal{A}$  surfaces  with $\dim\{\mathfrak{K}(\mathcal{M})\}\leq 3$.
Then $\mathcal{M}_1$ and $\mathcal{M}_2$ are affine isomorphic if and only if they are linearly isomorphic.
\end{lemma}

\begin{remark}\rm
There exist Type~$\mathcal{A}$ surfaces with $\dim\{\mathfrak{K}(\mathcal{M})\}=4$ which are not linearly equivalent
but which nevertheless are affine equivalent. We refer to the discussion in \cite{BGG16a} for further details.
\end{remark}

Let $\mathcal{M}=(\mathbb{R}^+\times\mathbb{R},\Gamma)$ be a Type~$\mathcal{B}$ affine surface model. The $a\vec x+b$ group preserves this geometry and two Type~$\mathcal{B}$ models $\mathcal{M}_1$ and $\mathcal{M}_2$ are said to be \emph{linearly isomorphic} if and only if there exists  an affine transformation of the form
$\Psi(x^1,x^2)=(x^1, a^2_1 x^1+ a^2_2 x^2)$ for $ a^2_2\ne0$ intertwining the two structures.

The following observation follows from the discussion  in \cite{BGG16}.
It is a non-trivial assertion as there are non-linear affine
transformations from one model to another if the dimension of the space of affine Killing vector fields is $3$-dimensional.
However, they play no role in defining the affine isomorphism type (see Lemma 3.17 and Theorem 3.21 in \cite{BGG16}).

\begin{lemma}\label{L2.4}
Let $\mathcal{M}_1$ and $\mathcal{M}_2$ be two Type~$\mathcal{B}$  surfaces which are not also of Type~$\mathcal{A}$, i.e. 
$\dim\{\mathfrak{K}(\mathcal{M})\}\leq 3$.
Then $\mathcal{M}_1$ and $\mathcal{M}_2$ are affine isomorphic if and only if they are linearly isomorphic.
\end{lemma}

Any Type~$\mathcal{C}$ geometry which is not modeled on a Type~$\mathcal{B}$ geometry is modeled on $\mathbb{S}^2$.

\subsection{The affine quasi-Einstein equation for homogeneous surfaces}\label{S2.2}
Let $\mathcal{M}$ be a homogeneous surface.
If $\mathcal{M}$ is of Type~$\mathcal{C}$, then  $\nabla$ is the Levi-Civita connection of a metric $g$ of constant sectional curvature $c$ and thus strongly projectively flat.  Theorem \ref{T1.3} shows that
$E(0,\nabla)=\operatorname{Span}\{ 1\}$, and  $E(\mu,\nabla)=0$ for $\mu\neq -1 ,0$  in this case.
Since $\rho^g=c g$, the affine quasi-Einstein equation~\eqref{E1.a} with $\mu=-1$ reduces to the Obata's equation $\mathcal{H} f+ c f g=0$ \cite{Kanai, obata}. Solutions to the Obata's equation are given by the eigenfunctions corresponding to the first eigenvalue of the Laplacian.

Let $\mathcal{M}$ be a Type~$\mathcal{A}$ or Type~$\mathcal{B}$ surface model.
It is convenient to complexify and allow complex valued functions;
we set $E_{\mathbb{C}}(\mu,\nabla):=E(\mu,\nabla)\otimes_{\mathbb{R}}\mathbb{C}$.
 As we shall always take $\mu$ real and as the underlying structures are real, elements of $E(\mu,\nabla)$ may be obtained by taking the
real and imaginary parts of complex solutions. In \cite{BGGV17b}, we used the structure of 
$E_{\mathbb{C}}(\mu,\nabla)$ as
a module over the affine Killing vector fields to obtain information about the structure of the eigenspaces $E(\mu,\nabla)$ showing:

\begin{theorem}\label{T2.5}
\ \begin{enumerate}
\item Let $\mathcal{M}=(\mathbb{R}^2,\Gamma)$ be a Type~$\mathcal{A}$ surface model. 
We can choose a spanning set for $E_{\mathbb{C}}(\mu,\nabla)$
consisting of functions $e^{\alpha_1x^1+\alpha_2x^2}p(x^1,x^2)$ where $p$ is polynomial. Furthermore, in
this setting $e^{\alpha_1x^1+\alpha_2x^2}\partial_{x^i}p\in E(\mu,\nabla)$ for $i=1,2$.
\item Let $\mathcal{M}=(\mathbb{R}^+\times\mathbb{R},\Gamma)$ be a Type~$\mathcal{B}$ surface model. Let $\mu\in\mathbb{R}$.
We can expand any element of $E(\mu,\nabla)$ as a finite sum
$f=\sum_\alpha(x^1)^{\alpha}p_\alpha(\log(x^1),x^2)$ where $p_{\alpha}$ is polynomial.
\end{enumerate}
\end{theorem}

Theorem~\ref{T2.5} gives the general form of a solution to the quasi-Einstein equation for Type~$\mathcal{A}$ and
Type~$\mathcal{B}$ surface models.  One can improve this generic result in the Type~$\mathcal{B}$ setting as follows:

\begin{theorem}\label{T2.6}
Let $\mathcal{M}$ be a non-flat Type~$\mathcal{B}$ surface model and let $f\in E(\mu,\nabla)$. Let  $c_i$, $c_{ij}$ be arbitrary constants.
\begin{enumerate}
\item We have one of the following possibilities:
\begin{enumerate}
\item Let $f(x^1,x^2)=(x^1)^\alpha\{(x^2)^2+2c_{12}x^1x^2+c_{11}(x^1)^2\}$. Then
$\mu=-1$ and $E(-1,\nabla)=\operatorname{Span}\{f,(x^1)^\alpha\{x^2+c_{12}x^1\},(x^1)^\alpha\}$.
\item Let $f(x^1,x^2)=(x^1)^\alpha\{x^2+c_1x^1\}$. Either
\begin{enumerate}
\item $E(\mu,\nabla)=\operatorname{Span}\{f,(x^1)^\alpha\}$.
\item $\mu=-1$ and $E(-1,\nabla)=\operatorname{Span}\{f,(x^1)^\alpha,(x^1)^\beta\}$ for $\beta\ne\alpha$ and $\beta\ne\alpha+1$.\
\end{enumerate}
\item Let $f(x^1,x^2)=(x^1)^\alpha\{x^2+c_1\log(x^1)x^1\}$ for $c_1\neq 0$. Then $\mu=-1$, and 
$
E(-1,\nabla)=\operatorname{Span}\{f,(x^1)^\alpha,(x^1)^{\alpha+1}\}\,.
$
\item Let $f(x^1,x^2)=f(x^1)$. We have the following possibilities:
\begin{enumerate}
\item $E(\mu,\nabla)=\operatorname{Span}\{(x^1)^\alpha\}$.
\item $E(\mu,\nabla)=\operatorname{Span}\{(x^1)^\alpha,(x^1)^\alpha\log(x^1)\}$.
\item $\mu=-1$ and $E(\mu,\nabla)=\operatorname{Span}\{(x^1)^\alpha,(x^1)^\alpha\log(x^1),(x^1)^\beta\}$ for $\alpha\ne\beta$.
\item $E(\mu,\nabla)=\operatorname{Span}\{(x^1)^\alpha,(x^1)^\beta\}$ for $\alpha\ne\beta$.
\item $\mu=-1$ and $E(\mu,\nabla)=\operatorname{Span}\{(x^1)^\alpha,(x^1)^\beta,(x^1)^\gamma\}$ for $\{\alpha,\beta,\gamma\}$
distinct.
\end{enumerate}
\end{enumerate}
\item The constants $c_i$, $c_{ij}$ are real in Assertion {\rm (1)}. In Assertion~{\rm (1d-iv)}, if $\alpha$ is complex, then $\beta=\bar\alpha$.
In Assertion~{\rm (1d-v)}, if $\alpha$ is complex, then we can choose the notation so $\beta=\bar\alpha$ and $\gamma$ is real.
In the remaining statements of Assertion~{\rm (1)}, $\alpha$ and $\beta$ are real.
\item If $(x^1)^{\alpha_1}\in E(\mu,\nabla)$ and $(x^1)^{\alpha_2}\in E(\mu,\nabla)$ for $\alpha_1\ne\alpha_2$, then 
\begin{eqnarray*}
&&C_{12}{}^1=0,\qquad C_{22}{}^1=0,\qquad C_{22}{}^2\mu=0,\\
&&\alpha_i^2-\alpha_i(1+C_{11}{}^1)-(C_{12}{}^2+C_{11}{}^1C_{12}{}^2-(C_{12}{}^2)^2)\mu=0\,.
\end{eqnarray*}
\end{enumerate}
\end{theorem}

\begin{remark}\rm In fact, all the possibilities of Assertion~(1)  in Theorem~\ref{T2.6} can occur.
	Assertion~(1a) is illustrated by Theorem \ref{T4.4} and by Theorem~\ref{T5.7}. Assertion~(1b) is illustrated by Theorem \ref{T4.4} and Theorem~\ref{T5.10}~(1). 
	Assertion~(1c) and Assertion~(1d) are illustrated by Theorem~\ref{T3.8} (see also Remark \ref{R3.5} for $\mu=0$ and $\alpha=0$). Every example where one has $\dim\{ E(\mu,\nabla)\}\ne0$ illustrates Assertion~(1e).
	
	Assertion~(3) is illustrated by Theorem~\ref{T3.8} where we present examples
	of Type~$\mathcal{B}$ surfaces which are also of Type~$\mathcal{A}$.
	Suppose that $E(\mu,\nabla)=\operatorname{Span}\{(x^1)^\alpha,(x^1)^\beta\}$ for $\alpha\ne\beta$. We can either have $\alpha$ and $\beta$ real
	or $\beta=\bar\alpha$; in this latter instance, real solutions can be obtained by taking the real and imaginary parts of $(x^1)^\alpha$.
	Assertion~(3) is also illustrated by  Theorem~\ref{T5.6}~(4) where we have $E(0,\nabla)=\operatorname{Span}\{1,(x^1)^{C_{11}{}^1+1}\}$ 
	and $C_{22}{}^2\ne0$.
\end{remark}

\begin{proof}
Let $X:=x^1\partial_{x^1}+x^2\partial_{x^2}$ and $Y:=\partial_{x^2}$. As $X$ and $Y$ are affine
Killing vector fields, $X\cdot E(\mu,\nabla)\subset E(\mu,\nabla)$ and $Y\cdot E(\mu,\nabla)\subset E(\mu,\nabla)$. 
By Theorem~\ref{T1.1}, $\dim\{E(\mu,\nabla)\}\le3$.  Since $\mathcal{M}$ is non-flat, by Theorem~\ref{T1.2}, $\dim\{E(\mu,\nabla)\}=3$
if and only if $\mu=-1$ and $\mathcal{M}$ is strongly projectively flat.
Let $f\in E(\mu,\nabla)$. By Theorem~\ref{T2.5}~(2), we may expand
$$
f=\sum_{\alpha,i,j}c_{\alpha,i,j}(x^1)^\alpha(x^2)^i\log(x^1)^j\,.
$$Choose $i_0$ maximal so there exists $f\in E(\mu,\nabla)$ with
$c_{\alpha,i_0,j}(f)\ne0$ for some $\alpha$ and some $j$. Then $f=\sum_{i\le i_0}f_i(x^1)(x^2)^i$ and $f_{i_0}(x^1)\ne0$. We have
\begin{equation}\label{E2.aa}\begin{array}{l}
Xf=\displaystyle\sum_{\alpha,i,j}c_{\alpha,i,j}(x^1)^\alpha(x^2)^i\{(\alpha+i)\log(x^1)^j+j\log(x^1)^{j-1}\}\in E(\mu,\nabla),\\[0.05in]
Yf=\displaystyle\sum_{\alpha,i,j}ic_{\alpha,i,j}(x^1)^\alpha(x^2)^{i-1}\log(x^1)^j\in E(\mu,\nabla).
\end{array}\end{equation}
We establish Assertion~1 by checking various cases. 

\subsection*{Case 1.} Suppose $i_0 \ge2$.  By Equation~(\ref{E2.aa}),
$\{f,Yf,\dots,Y^if\}$ are $i+1$ linearly independent elements of $E(\mu,\nabla)$. As $\dim\{E(\mu,\nabla)\}\le3$,
$i_0=2$ and
\begin{equation}\label{E2.a}
\begin{array}{l}
E(\mu,\nabla)=\operatorname{Span}\{f=f_2(x^1)(x^2)^2+f_1(x^1)x^2+f_0(x^1),\\[0.05in]
\qquad\qquad\qquad\quad Yf=2f_2(x^1)x^2+f_1(x^1),\ \ Y^2f=2f_2(x^1)\}\,.
\end{array}\end{equation}
Since $Xf=((X+2)f_2(x^1))(x^2)^2+((X+1)f_1(x^1))x^2+Xf_0(x^1)$
 is a linear combination of $\{f,Yf, Y^2f\}$, $(X+2)f_2$ is a multiple of $f_2$ so
we may assume $f_2(x^1)=(x^1)^\alpha$ for some $\alpha$. Consequently, $Y^2f=2(x^1)^\alpha\in E(\mu,\nabla)$. We have
$$
Y(X-\alpha-2)f=(X-\alpha-1)f_1 (x^1)\in E(\mu,\nabla)\,.
$$
This must be a multiple of $(x^1)^\alpha$ and hence 
$f_1(x^1)=c_1(x^1)^\alpha+2c_{12}(x^1)^{\alpha+1}$.
We subtract $\frac12c_1Yf$ to ensure that $c_1=0$ and express
$$
f=(x^1)^\alpha\{(x^2)^2+2c_{12}x^1x^2\}+f_0(x^1)\,.
$$
We have $(X-\alpha-2)f=(X-\alpha-2)f_0$. Thus $(X-\alpha-2)f_0$ is a multiple of $(x^1)^\alpha$. 
Consequently, $f_0(x^1)=c_{11}(x^1)^{\alpha+2}+c_3(x^1)^\alpha$. Subtract an appropriate multiple of $(x^1)^\alpha$
to assume $c_3=0$ and obtain $f=(x^1)^\alpha\{(x^2)^2+2c_{12}x^1x^2+c_{11}(x^1)^2\}$. This is the possibility of Assertion~(1a).

\subsection*{Case 2} Assume $i_0=1$ so there exists a function 
$f=f_1(x^1)x^2+f_0(x^1)\in E(\mu,\nabla)$ which is linear in $x^2$. Suppose $\{f_1,Xf_1\}$ are linearly independent.
Then
$$\begin{array}{lll}
f=f_1(x^1)x^2+f_0(x^1),&Xf=((X+1)f_1(x^1))x^2+Xf_0(x^1),\\[0.05in]
Yf=f_1(x^1),&YXf=(X+1)f_1(x^1)
\end{array}$$
are 4 linearly independent elements of $E(\mu,\nabla)$; this is false. Thus $Xf_1$ is a multiple of $f_1$ so we may assume
$f_1=(x^1)^\alpha$ for some $\alpha$. Consequently $f=(x^1)^\alpha x^2+f_0(x^1)$ so
$(x^1)^\alpha=Yf\in E(\mu,\nabla)$.
Then $(X-1-\alpha)f=(X-1-\alpha)f_0$.
Several cases present themselves.
\begin{enumerate}
\item $E(\mu,\nabla)=\operatorname{Span}\{f=(x^1)^\alpha x^2+f_0(x^1),(x^1)^\alpha\}$. We have that
 $(X-1-\alpha)f_0$ is
a multiple of $(x^1)^\alpha$. This shows that
$f_0(x^1)=c(x^1)^{\alpha+1}+c_1(x^1)^\alpha$. By subtracting a multiple of $(x^1)^\alpha$,
we obtain the form of Assertion~(1b-i)
$f=(x^1)^\alpha x^2+c_1(x^1)^{\alpha+1}=(x^1)^\alpha\{x^2+c_1x^1\}$.
\item $E(\mu,\nabla)=\operatorname{Span}\{f,(x^1)^\alpha,(x^1)^\beta\}$ for $\alpha\ne\beta$ and $\beta\ne\alpha+1$.
We now have $(X-\alpha-1)f_0=d_1x^\alpha+d_2x^\beta$ implies $f_0=c_1(x^1)^{\alpha+1}+\tilde d_1x^\alpha+\tilde d_2x^\beta$.
By subtracting appropriate multiples of $(x^1)^\beta$ and $(x^1)^\alpha$, we may assume that $\tilde d_i=0$ and $f=(x^\alpha)(x^2+c_1x^1)$ and obtain Assertion~(1b-ii).
\item 
$E(\mu,\nabla)=\operatorname{Span}\{f,(x^1)^\alpha,(x^1)^\beta\}$ for $\alpha\ne\beta$ and
 $\beta=\alpha+1$. We now obtain $(X-\alpha-1)f_0$ is a linear combination of $(x^1)^\alpha$ and $(x^1)^{\alpha+1}$.
Thus
$f_0=d_1(x^1)^\alpha+d_2(x^1)^{\alpha+1}+c_1\log(x^1)(x^1)^{\alpha+1}$ and we obtain Assertion~(1c).
\end{enumerate}

\smallbreak\noindent{\bf Case 3.} We have $i_0=0$ so every function in $E(\mu,\nabla)$ only depends on $x^1$. 
We may decompose $E(\mu,\lambda)$
as a direct sum of the generalized eigenspaces of $X$; a function $f$ belongs to such an eigenspace if and only
if $(X-\alpha)^3f=0$ for some $\alpha\in\mathbb{C}$. Such a function satisfies
$f=(x^1)^\alpha\{c_0+c_1\log(x^1)+c_2\log(x^1)^2\}$.
Suppose $c_2\ne0$. Since $(X-\alpha)^2f=2c_2(x^1)^\alpha$, we see $(x^1)^\alpha\in E(\mu,\nabla)$. Thus renormalizing,
we obtain $(x^1)^\alpha\{c_1\log(x^1)+c_2\log(x^1)^2\}\in E(\mu,\nabla)$. A similar argument, shows 
$c_2(x^1)^\alpha\log(x^1)\in E(\mu,\nabla)$ and thus finally $c_2\log(x^1)^2\in E(\mu,\nabla)$. This is not possible as this
function vanishes to second order at $x=1$ and that violates Theorem~\ref{T1.1}. Thus we have $c_2=0$ and
$E(\mu,\nabla)$ is the span of functions of the form $(x^1)^\alpha$ or $(x^1)^\alpha\log(x^1)$. This establishes Assertion~(1).

Assertion~(2) follows from the arguments we gave to establish Assertion~(1).  
Since Equation~(\ref{E1.a}) is a real PDE, we can take the
complex conjugate of any solution. If $\alpha\in\mathbb{C}-\mathbb{R}$, then taking the real and imaginary part of $(x^1)^\alpha$ and
the solutions in Assertions~(1) would give too many solutions except in the instances noted; for the same reason, the constants
must be real.

To prove Assertion~(3), suppose $(x^1)^\alpha\in E(\mu,\nabla)$ and $(x^1)^\beta\in E(\mu,\nabla)$ where $\alpha\ne\beta$. 
We introduce the affine quasi-Einstein operator $\mathfrak{Q}_\mu(f):=\mathcal{H}f-\mu\, f\rho_s$
for any $f\in C^\infty(M)$ (see \cite{BGGV17b}) and 
compute
\begin{eqnarray*}
&&(x^1)^{2-\alpha}\mathfrak{Q}_\mu ((x^1)^\alpha)-(x^1)^{2-\beta}\mathfrak{Q}_\mu ((x^1)^\beta)\\
&=&(\alpha-\beta)\left(\begin{array}{cc}-1+\alpha+\beta-C_{11}{}^1&-C_{12}{}^1\\-C_{12}{}^1&-C_{22}{}^1\end{array}\right)\,.
\end{eqnarray*}
This implies $C_{12}{}^1=C_{22}{}^1=0$. We impose these relations and set $\mathfrak{Q}_\mu ((x^1)^\alpha)=0$ to obtain the final two relations.
\end{proof}

\section{Homogeneous surfaces with $\dim\{\mathfrak{K}(\mathcal{M})\}=4$}\label{S3}

Results in \cite{BGG16} show that a surface has  $\dim\{\mathfrak{K}(\mathcal{M})\}=4$ if and only if $\mathcal{M}$ is 
strongly projectively flat and recurrent with Ricci tensor of rank one. Moreover, there exist representatives of 
Type~$\mathcal{A}$ for all affine isomorphic classes, i.e.
any surface of Type~$\mathcal{B}$ with $\dim\{\mathfrak{K}(\mathcal{M})\}=4$ is locally isomorphic to a Type~$\mathcal{A}$ surface
with $\dim\{\mathfrak{K}(\mathcal{M})\}=4$. We treat the $\mathcal{A}$ and the $\mathcal{B}$ geometries separately since they
give rise to different local structures.

\subsection{Type~$\mathcal{A}$ surface models with $\dim\{\mathfrak{K}(\mathcal{M})\}=4$}
In this section, we shall assume that $\mathcal{M}$ is a homogeneous affine surface 
of Type~$\mathcal{A}$; $\dim\{\mathfrak{K}(\mathcal{M})\}=4$ if and only if $\operatorname{Rank}\{\rho\}=1$.
We impose this condition hence forth and make a linear change of coordinates to ensure $\rho=\rho_{22}dx^2\otimes dx^2$.

\begin{definition}\label{D3.1}
Define the following  Type~$\mathcal{A}$ surface models with $\rho=\rho_{22}dx^2\otimes dx^2$:
\medbreak\quad $\mathcal{M}_1$:
$\Gamma_{11}{}^1=-1$, $\Gamma_{11}{}^2=0$, $\Gamma_{12}{}^1=1$, $\Gamma_{12}{}^2=0$, 
$\Gamma_{22}{}^1=0$, $\Gamma_{22}{}^2=2$, $\rho_{22}=1$.
\smallbreak\quad $\mathcal{M}_2^c$:
$\Gamma_{11}{}^1=-1$, $\Gamma_{11}{}^2=0$, 
$\Gamma_{12}{}^1=c$, $\Gamma_{12}{}^2=0$, $\Gamma_{22}{}^1=0$,
$\Gamma_{22}{}^2=1+2c$,\smallbreak\qquad\quad\   $\rho_{22}=c^2+c\ne0$.
\smallbreak\quad $\mathcal{M}_3^c$: $\Gamma_{11}{}^1=0$, $\Gamma_{11}{}^2=0$,
$\Gamma_{12}{}^1=c$, $\Gamma_{12}{}^2=0$, 
$\Gamma_{22}{}^1=0$,
$\Gamma_{22}{}^2=1+2c$,\smallbreak\qquad\quad\  $\rho_{22}=c^2+c\ne0$.
\smallbreak\quad $\mathcal{M}_4^c$: $\Gamma_{11}{}^1=0$, 
$\Gamma_{11}{}^2=0$, $\Gamma_{12}{}^1=1$,
$\Gamma_{12}{}^2=0$, $\Gamma_{22}{}^1=c$, $\Gamma_{22}{}^2=2$, $\rho_{22}=1$. 
\smallbreak\quad $\mathcal{M}_5^c:\Gamma_{11}{}^1=-1$, 
$\Gamma_{11}{}^2=0$, $\Gamma_{12}{}^1=c$,
$\Gamma_{12}{}^2=0$,
$\Gamma_{22}{}^1=-1$, $\Gamma_{22}{}^2=2c$,\smallbreak\qquad\quad\    $\rho_{22}=1+c^2$. 
\end{definition}

\begin{remark}
\rm
The models $\mathcal{M}_2^c$ and $\mathcal{M}_3^c$ are affine isomorphic and the models $\mathcal{M}_1$ and $\mathcal{M}_4^c$ are affine isomorphic for any value of $c$.
	On the other hand, no surface in one family of Definition~\ref{D3.1} is linearly isomorphic to a surface in another family \cite{BGG16}. 
	There are, however, some linear isomorphisms amongst the surfaces in each family. 
	For example, one can rescale $x^1$ to see that $\mathcal{M}_4^c$ and $\mathcal{M}_4^{\tilde c}$ are
	linearly isomorphic if $c\ne0$ and $\tilde c\ne0$. Furthermore, $\mathcal{M}_5^c$ and $\mathcal{M}_5^{-c}$ are linearly isomorphic (see \cite{RBGGPV18}).
\end{remark}

We have the following classification result \cite{BGG16}:
\begin{theorem}
Let $\mathcal{M}$ be a homogeneous Type~$\mathcal{A}$ surface model with $\dim\{\mathfrak{K}(\mathcal{M})\}=4$. 
Then $\mathcal{M}$ is affine isomorphic to $\mathcal{M}_1$, a $\mathcal{M}_2^c$ model or  a $\mathcal{M}_5^c$ model.
\end{theorem}

Since we assume $\rho_s\neq 0$,
one has $1\in E(\mu,\nabla)$ if and only if $\mu=0$. Moreover, $1\in E(0,\nabla)$ for any affine surface and thus $\dim\{E( 0,\nabla)\}\geq 1$.

\begin{lemma}\label{L3.4}
Let $\mathcal{M}$ be a  surface with $\dim\{\mathfrak{K}(\mathcal{M})\}=4$. Then $\dim\{E(0,\nabla)\}\geq 2$ if and only if one of the following possibilities pertains:
\begin{enumerate}
\item $\mathcal{M}$ is affine isomorphic to $\mathcal{M}_5^{\pm\frac{1}{2}}$, and $E(0,\nabla)=\operatorname{Span}\{1,e^{x^1}\}$.
\smallbreak
\item $\mathcal{M}$ is affine isomorphic to $\mathcal{M}_5^0$, or to $\mathcal{M}_2^{-\frac{1}{2}}$, and
$E(0,\nabla)=\operatorname{Span}\{1,x^1\}$. 
\end{enumerate}
\end{lemma}

\begin{remark}\label{R3.5}
\rm
Among the surfaces given in Lemma~\ref{L3.4}, only $\mathcal{M}_2^{-\frac{1}{2}}$ ($\cong\mathcal{M}_3^{-\frac{1}{2}}$) is also Type~$\mathcal{B}$, which is the only symmetric Type~$\mathcal{B}$ surface \cite{sym}. Thus a Type~$\mathcal{B}$ model with $\dim\{\mathfrak{K}(\mathcal{M})\}=4$ has $\dim\{E(0,\nabla)\}\geq 2$ if and only if up to  linear equivalence:
$C_{12}{}^1=0$, $C_{22}{}^1=0$, $C_{22}{}^2=0$, and $C_{11}{}^1=-1$. In this case 
we have that $E(0,\nabla)=\operatorname{Span}\{ 1, \log(x^1)\}$.
\end{remark}

\begin{proof}  
Theorem~\ref{T2.5} shows that, for any Type~$\mathcal{A}$ model, $E(\mu,\nabla)$ is spanned by elements of the form 
$f(\vec x)=p(\vec x)e^{\vec a\cdot\vec x}$.
First note that, since the coordinate vector fields are affine Killing vector fields and
$(\partial_{x^i}-a_i)f=(\partial_{x^i}p)e^{\vec a\cdot\vec x}$, one has $(\partial_{x^i}p)e^{\vec a\cdot\vec x}\in E(\mu,\nabla) $ for $i=1,2$. 
Suppose there exists a non-constant function $f\in E(0,\nabla)$, i.e. $\dim\{E(0,\nabla)\}\geq 2$. 
By Lemma~4.1 of \cite{BGG16}, $R_{ij}(df)=0$.
This implies that $df$ belongs to the kernel of the curvature operator. Consequently after a suitable linear change of coordinates, we have
$f=f(x^1)$ for any $f\in E(0,\nabla)$. Suppose first that $e^{a_1x^1}\in E(0,\nabla)$  for $a_1\ne0$. 
Since $\dim\{E(0,\nabla)\}\ne3$, $a_1$ is real. By rescaling coordinates appropriately, we can assume $a_1=1$.  
Equation~(\ref{E1.a}) implies
$\Gamma_{11}{}^1=1$, $\Gamma_{12}{}^1=0$, and $\Gamma_{22}{}^1=0$.
A direct calculation shows that
\begin{eqnarray*}
&&\rho=(\Gamma_{12}{}^2(1-\Gamma_{12}{}^2)+\Gamma_{11}{}^2\Gamma_{22}{}^2)\, dx^1\otimes dx^1,\\
&&\nabla\rho=-2(\Gamma_{12}{}^2(1-\Gamma_{12}{}^2)+\Gamma_{11}{}^2\Gamma_{22}{}^2)\, dx^1\otimes dx^1\otimes dx^1\,.
\end{eqnarray*}

The affine equivalence classes of Type~$\mathcal{A}$ surfaces with Ricci tensor of rank one are parametrized by two invariants (see Theorem 3.8 in \cite{BGG16}):
$$
\alpha_X(\mathcal{M})=\nabla\rho(X,X;X)^2\rho(X,X)^{-3}\quad \mbox{and}\quad
\varepsilon_X(\mathcal{M})=\operatorname{Sign}\{\rho(X,X)\}=\pm 1.
$$
Now a direct calculation shows that
$\alpha_X(\mathcal{M})= 4\rho(X,X)^{-1}$, which is only possible if $\mathcal{M}$ is affine isomorphic to $\mathcal{M}_5^c$ with $c=\pm\frac{1}{2}$ as in Assertion~(1). 
A direct computation then shows $E(0,\nabla)=\operatorname{Span}\{1,e^{x^1}\}$.
 
Suppose next that $x^1$ solves Equation~(\ref{E1.a}). We then obtain
$\Gamma_{11}{}^1=0$, $\Gamma_{12}{}^1=0$, and $\Gamma_{22}{}^1=0$  and thus $\nabla\rho=0$. Hence Theorem 3.8 in \cite{BGG16} shows that  $\mathcal{M}$ is affine isomorphic to $\mathcal{M}_5^0$, or to  $\mathcal{M}_2^{c}$ with $c=-\frac{1}{2}$ as in Assertion~(2).  Again, a direct computation then implies $E(0,\nabla)=\operatorname{Span}\{1,x^1\}$. 
\end{proof}

We now assume $\mu\ne0$ and that $f(x^1,x^2):=e^{a_1x^1+a_2x^2}$ defines an element of $E(\mu,\nabla)$.
Since $1\notin E(\mu,\nabla)$, $(a_1,a_2)\ne(0,0)$.
We shall examine what values of $(a_1,a_2)$ are possible and also what linear expressions $(b_1x^1+b_2x^2)f$ are possible; there
are no quadratic expressions if $\mu\neq-1$ by Theorem \ref{T2.5} and Theorem~\ref{T1.1} since $\operatorname{Rank}\{\rho\}=1$
implies $\mathcal{M}$ is not flat. In each case, a direct computation exhibits
3 linearly independent elements of $E(-1,\nabla)$ so $E(-1,\nabla)$ is spanned by these elements for dimensional reasons; the case $\mu=-1$ is exceptional
and is dealt with separately. In what follows we examine the families of Definition~\ref{D3.1} seriatim; they provide rich examples of solutions. 
We set
$$
f(x^1,x^2):=e^{a_1x^1+a_2x^2}\text{ and }\mathfrak{Q}_\mu(f):=\mathcal{H}f-\mu\, f\rho_s\,.
$$
 
\subsection*{Case 1. The manifold $\mathcal{M}_1$.}
We compute:
$$\begin{array}{ll}
\mathfrak{Q}_{\mu,11} (f)=a_1 (a_1+1)f,&\mathfrak{Q}_{\mu,12}(f)= a_1 (a_2-1)f,\\[0.05in]
\mathfrak{Q}_{\mu,22}(f)=\{a_2^2-2 a_2-\mu \}f\,.
\end{array}$$

\begin{enumerate}
\item $E(-1,\nabla)=\operatorname{Span}\{e^{-x^1+x^2},e^{x^2},x^2e^{x^2}\}$.
\item Let $\mu\ne-1$. Setting $\mathfrak{Q}_{\mu,22}(f)=0$ yields a quadratic relation for $\mu$ with associated
discriminant $4+4\mu\ne0$. If $a_2=1\pm \sqrt{1+\mu}$, then $\mathfrak{Q}_{\mu,22}(f)=0$. Since $\mu\neq -1$, $a_2\neq 1$.
Consequently $a_1=0$. There
are no solutions of the form $p(\vec x)f(\vec x)$ for $\operatorname{deg}\{p\}\geq 1$ so
$$E(\mu,\nabla)=\operatorname{Span}\{e^{(1- \sqrt{1+\mu})x^2},e^{(1+ \sqrt{1+\mu})x^2}\}\,.$$
\end{enumerate}

\subsection*{Case 2. The manifold $\mathcal{M}_2^c$.} We have $c^2+c\ne0$ to ensure $\rho\ne0$. We compute
\begin{eqnarray*}
&&\mathfrak{Q}_{\mu,11}(f)=a_1 (a_1+1) f,\quad \mathfrak{Q}_{\mu,12}(f)= a_1 (a_2-c)f,\\
&&\mathfrak{Q}_{\mu,22}(f)=\{a_2^2-(2 c+1) a_2-c (c+1) \mu\}f\,.
\end{eqnarray*}
\begin{enumerate}
\item $E(-1,\nabla)=\operatorname{Span}\{e^{cx^2},e^{(1+c)x^2},e^{-x^1+cx^2}\}$.
\item  Let $\mu\ne-1$. If  $a_2=c$, then $\mathfrak{Q}_{\mu,22}f=c(1+c)(1+\mu)f\ne0$. Consequently,
$a_2\ne c$ so $a_1=0$. Setting $\mathfrak{Q}_{\mu,22}(f)= 0$ yields a quadratic relation for $a_2$; the critical eigenvalue $\mu_0:=-\frac{(1+2c)^2}{4c(1+c)}$ annihilates the discriminant.
\begin{enumerate}
\smallbreak\item If $\mu\ne\mu_0$,  the relation $\mathfrak{Q}_{\mu,22}(f)=0$ has two distinct solutions $a_2=s_i$, $i=1,2$, and
$E(\mu,\nabla)=\operatorname{Span}\{e^{s_1x^2},e^{s_2x^2}\}$; there are no solutions $p(\vec x)f(\vec x)$ for $\operatorname{deg}\{p\}\geq 1$.
\smallbreak\item If $\mu=\mu_0$, then $a_2=c+\frac12$ and
$E(\mu,\nabla)=\operatorname{Span}\{e^{a_2x^2},x^2e^{a_2x^2}\}$. 
\end{enumerate}\end{enumerate}

\subsection*{Case 3. The manifold $\mathcal{M}_3^c$.} This model is non-linearly locally isomorphic to $\mathcal{M}_2^c$. Thus the
dimension of $E(\mu,\nabla)$ is unchanged, although the local form of the functions can change. We have
\begin{eqnarray*}
&&\mathfrak{Q}_{\mu,11}(f)=a_1^2f,\qquad\mathfrak{Q}_{\mu,12}(f)=a_1(a_2-c)f,\\
&&\mathfrak{Q}_{\mu,22}(f)=(a_2^2-a_2(1+2c)-c(1+c)\mu)f\,.
\end{eqnarray*}
\begin{enumerate}
\item $E(-1,\nabla)=\operatorname{Span}\{e^{cx^2},x^2e^{cx^2},e^{(1+c)x^2}\}$.
\item Let $\mu\ne-1$.  We have $a_1=0$.  The quadratic relation for $a_2$ obtained by setting $\mathfrak{Q}_{\mu,22}(f)$ is the
same as for $\mathcal{M}_2^c$. The remainder of the analysis is the same. Let $\mu_0:=-\frac{(1+2c)^2}{4c(1+c)}$ be the critical eigenvalue annihilating the discriminant.
\begin{enumerate}
\smallbreak\item If $\mu\ne\mu_0$,  the relation $\mathfrak{Q}_{\mu,22}(f)=0$ has two distinct solutions $a_2=s_i$, $i=1,2$, and
$E(\mu,\nabla)=\operatorname{Span}\{e^{s_1x^2},e^{s_2x^2}\}$.
\smallbreak\item If $\mu=\mu_0$, then $a_2=c+\frac12$ and
$E(\mu,\nabla)=\operatorname{Span}\{e^{a_2x^2},x^2e^{a_2x^2}\}$. 
\end{enumerate}\end{enumerate}

\subsection*{Case 4. The manifold $\mathcal{M}_4^c$.} These models are all non-linearly isomorphic to $\mathcal{M}_1$ and thus
$\dim\{E(\mu,\nabla)\}$ is unchanged. We compute
\begin{eqnarray*}
&&\mathfrak{Q}_{\mu,11}f=a_1^2f,\qquad\mathfrak{Q}_{\mu,12}f=a_1(a_2-1)f,\\
&&\mathfrak{Q}_{\mu,22}f=(a_2^2-2a_2-a_1c-\mu)f\,.
\end{eqnarray*}
\begin{enumerate}
\item $E(-1,\nabla)=\operatorname{Span}\{e^{x^2},x^2e^{x^2},(2x^1+c(x^2)^2)e^{x^2}\}$. 
\item Let $\mu\ne-1$. We conclude $a_1=0$ and we have $a_2^2-2a_2-\mu=0$. The discriminant of this relation is
$4+4\mu\ne0$. Thus there are two distinct solutions $a_2=1\pm \sqrt{1+\mu}$; there are no linear solutions $p(\vec x)f(\vec x)$ with $\operatorname{deg}\{p\}\geq 1$ and
$E(\mu,\nabla)=\operatorname{Span}\{e^{(1- \sqrt{1+\mu})x^2},e^{(1+ \sqrt{1+\mu})x^2}\}$.
\end{enumerate}

The manifolds $\mathcal{M}_4^c$ for $c\ne0$ are the only affine surfaces with $\operatorname{Rank}\{\rho\}=1$
which admit an element of $f\in E(-1,\nabla)$ where $f(\vec x)=p(\vec x)e^{\vec a\cdot\vec x}$ and $p(\vec x)$ is quadratic.

\subsection*{Case 5. The manifold $\mathcal{M}_5^c$.} This model is not Type~$\mathcal{B}$. We obtain
\begin{eqnarray*}
&&\mathfrak{Q}_{\mu,11}f=a_1(1+a_1)f,\qquad\mathfrak{Q}_{\mu,12}(f)=a_1(a_2-c)f,\\
&&\mathfrak{Q}_{\mu,22}(f)=(a_2^2 - 2 a_2 c+a_1 -(1 + c^2) \mu)f\,.
\end{eqnarray*}
\begin{enumerate}
\item $E(-1,\nabla)=\operatorname{Span}\{e^{-x^1+cx^2},e^{(c+\sqrt{-1})x^2},e^{(c-\sqrt{-1})x^2}\}$.
\item Let $\mu\ne-1$. If $a_1=-1$, then $a_2=c$ and $\mathfrak{Q}_{\mu,22}(f)=-(1+c^2)(1+\mu)f\ne0$. Thus $a_1\ne-1$ so $a_1=0$ and
we have $a_2^2-2a_2c-(1+c^2)\mu=0$. Setting the discriminant to zero yields $\mu_0=-\frac{c^2}{1+c^2}$. If $\mu=\mu_0$,
then  $a_2=c$, and $E(\mu_0,\nabla)=\operatorname{Span}\{e^{cx^2},x^2e^{cx^2}\}$. If $\mu\ne\mu_0$, we obtain two solutions
$a_2=s_1$ and $a_2=s_2$ for $s_1\ne s_2$. There are no other solutions to the equation and $E(\mu,\nabla)=\operatorname{Span}\{e^{s_1x^2},e^{s_2x^2}\}$.
\end{enumerate}

\subsection{Type~$\mathcal{B}$ surfaces with $\mathfrak{K}(\mathcal{M})=4$}

Results of \cite{BGG16} show that a Type~$\mathcal{B}$ surface model satisfies 
$\dim\{\mathfrak{K}(\mathcal{M})\}=4$ if andonly if it is also of
Type~$\mathcal{A}$. And furthermore, that either of these
two conditions is equivalent to the condition $\Gamma_{12}{}^1=0$, $\Gamma_{22}{}^1=0$, and $\Gamma_{22}{}^2=0$. Although
we have studied the solution space of Type~$\mathcal{A}$ geometries with $\dim\{\mathfrak{K}(\mathcal{M})\}=4$,
we shall also examine the Type~$\mathcal{B}$ geometries since we are interested in the exact form of the solutions.

\begin{definition}\label{D3.6}
Define the following Type~$\mathcal{B}$ surface models on $\mathbb{R}^+\times\mathbb{R}$:
$$
\begin{array}{ll}
{\mathcal Z}_1^\kappa:&
C_{11}{}^1=2\kappa,\, C_{11}{}^2=1,\, C_{12}{}^1=0,\, C_{12}{}^2=\kappa, \,
C_{22}{}^1=0,\, C_{22}{}^2=0.
\\
\noalign{\medskip}
{\mathcal Z}_2^{\kappa,\theta}:&
C_{11}{}^1=2\kappa+\theta-1,\, C_{11}{}^2=0,\, C_{12}{}^1=0,\, C_{12}{}^2=\kappa, \,
C_{22}{}^1=0,\, C_{22}{}^2=0.
\\
\noalign{\medskip}
{\mathcal Z}_3^\kappa: &
C_{11}{}^1=2\kappa-1,\, C_{11}{}^2=0,\, C_{12}{}^1=0,\, C_{12}{}^2=\kappa, \,
C_{22}{}^1=0,\, C_{22}{}^2=0.
\end{array}
$$
These geometries satisfy $\rho=(x^1)^{-2}\tilde\rho_{11}dx^1\otimes dx^1$.
\end{definition}

\begin{remark}\rm
Surfaces ${\mathcal Z}_1^\kappa$ satisfy
$\tilde\rho_{11}=\kappa(1+\kappa)$ and $\alpha_X({\mathcal Z}_1^\kappa)=\frac{4(2\kappa+1)^2}{\kappa(\kappa+1)}$, while surfaces ${\mathcal Z}_2^{\kappa,\theta}$ satisfy $\tilde\rho_{11}=\kappa(\kappa+\theta)$ and $\alpha_X({\mathcal Z}_2^{\kappa,\theta})=\frac{4(2\kappa+\theta)^2}{\kappa(\kappa+\theta)}$. Note that  ${\mathcal Z}_1^\kappa$ and $\mathcal{Z}_2^{\kappa,1}$ are affine isomorphic.  Furthermore ${\mathcal Z}_1^\kappa$ and $\mathcal{Z}_2^{\kappa,\theta}$ are affine isomorphic to $\mathcal{M}_2^c\cong\mathcal{M}_3^c$ for suitable $c$.
The surfaces ${\mathcal Z}_3^\kappa$ satisfy $\tilde\rho_{11}=\kappa^2$ and $\alpha_X({\mathcal Z}_3^\kappa)=16$. Hence ${\mathcal Z}_3^\kappa$ is affine isomorphic to $\mathcal{M}_1\cong\mathcal{M}_4^c$.
These affine isomorphisms are not linear (see Remark 3.16 in~\cite{BGG16}).
\end{remark}

We have the following classification result (which clarifies the statement of Theorem 3.11 in \cite{BGG16}):

\begin{theorem}\label{T3.8}
Let $\mathcal{M}$ be a Type~$\mathcal{B}$ surface model with $\dim\{\mathfrak{K}(\mathcal{M})\}=4$. Then one of the following holds
\begin{enumerate}
\item $\mathcal{M}$ is linearly isomorphic to $\mathcal{Z}_1^\kappa$, and\newline
$E(-1,\nabla)=\operatorname{Span}(x^1)^\kappa\{1, x^1, x^2+x^1\log x^1\}$. Moreover
\begin{enumerate}
\item If $\frac{1}{4}+\tilde{\rho}_{11}(1+\mu)\neq 0$, then $E(\mu,\nabla)=\operatorname{Span}\{ (x^1)^{\alpha_+}, (x^1)^{\alpha_-} \}$,
where $\alpha_\pm =\frac{1}{2}+\kappa\pm \sqrt{\frac{1}{4}+\tilde{\rho}_{11}(1+\mu)}$.
\item If $\frac{1}{4}+\tilde{\rho}_{11}(1+\mu)=0$, then
$\mu=-\frac{1}{4}(1+4\tilde{\rho}_{11}) \tilde{\rho}_{11}{}^{-1}$ and\newline
$E(\mu,\nabla)=\operatorname{Span}(x^1)^{\kappa+\frac{1}{2}}\{ 1, \log x^1 \}$.
\end{enumerate}

\item $\mathcal{M}$ is linearly isomorphic to $\mathcal{Z}_2^{\kappa,\theta}$, and 
$E(-1,\nabla)=\operatorname{Span}(x^1)^\kappa\{1, x^2,  (x^1)^{\theta} \}$. Moreover
\begin{enumerate}
\item If $(2\kappa+\theta)^2+4\mu\tilde{\rho}_{11}\neq 0$, then $E(\mu,\nabla)=\operatorname{Span}\{ (x^1)^{\alpha_+}, (x^1)^{\alpha_-} \}$,
where $\alpha_\pm =\frac{1}{2}\left(2\kappa+\theta\pm \sqrt{(2\kappa+\theta)^2+4\mu\tilde{\rho}_{11}}\right)$
\item If $(2\kappa+\theta)^2+4\mu\tilde{\rho}_{11}\neq 0$, then $\mu=-\frac{1}{4}(2\kappa+\theta)^2\tilde{\rho}_{11}{}^{-1}$ and\newline
$E(\mu,\nabla)=\operatorname{Span}(x^1)^{\kappa+\frac{1}{2}\theta}\{ 1, \log x^1 \}$.
\end{enumerate}

\item $\mathcal{M}$ is linearly isomorphic to $\mathcal{Z}_3^\kappa$, and 
$E(-1,\nabla)=\operatorname{Span}(x^1)^\kappa\{1, x^2,  \log x^1 \}$. Moreover
$E(\mu,\nabla)=\operatorname{Span}\{ (x^1)^{\alpha_+}, (x^1)^{\alpha_-} \}$,
where  $\alpha_\pm =\kappa\pm \sqrt{\tilde{\rho}_{11}(1+\mu)}$.
\end{enumerate}
\end{theorem}

\begin{proof} 
The fact that $\mathcal{M}$ is linearly isomorphic to one of the $\mathcal{Z}_i ^\kappa$ was shown in~\cite{BGG16}.
We make a direct computation to show that
the functions in Theorem~\ref{T3.8} satisfy Equation~(\ref{E1.a}) with the indicated eigenvalue $\mu$ under
the given conditions. Furthermore, the listed functions are linearly independent given the imposed relations.
By Theorem~\ref{T1.2}, if $\mathcal{M}$ is Type~$\mathcal{A}$, then
$\dim\{E(-1,\nabla)\}=3$ and by the analysis of the families in Definition~\ref{D3.6} we know that  \
$\dim\{E(\mu,\nabla)\}=2$ for $\mu\ne-1$. The fact that the given functions form a basis now follows.
\end{proof}

\section{Homogeneous surfaces with $\dim\{\mathfrak{K}(\mathcal{M})\}=3$}\label{S4}

A pseudo-Riemannian surface is of non-zero constant sectional curvature if and only if its isometry group is $3$-dimensional. Hence $\dim\{\mathfrak{K}(\mathcal{M})\}=3$ for any non-flat Type~$\mathcal{C}$ model. On the other hand there are  homogeneous Type~$\mathcal{B}$ models with $\dim\{\mathfrak{K}(\mathcal{M})\}=3$. 
In this section, we complete the study of the space of solutions to the affine quasi-Einstein equation for those affine surfaces with  $\dim\{\mathfrak{K}(\mathcal{M})\}=3$ by considering both types separately.

\subsection{Homogeneous models of Type~$\mathcal{B}$ with $\boldsymbol{\dim\{\mathfrak{K}(\mathcal{M})\}=3}$}

\begin{definition}\label{D4.1}\rm Define the following Type~$\mathcal{B}$ structures for $c\ge0$:
\smallbreak$\mathcal{N}_1^\pm:
C_{11}{}^1=-\frac32$, $C_{11}{}^2=0$, $C_{12}{}^1=0$, 
$C_{12}{}^2=-\frac12$, $C_{22}{}^1=\mp\frac12$,
$C_{22}{}^2=0$.
\smallbreak $\mathcal{N}_2^c:C_{11}{}^1=-\frac32$, 
$C_{11}{}^2=0$, $C_{12}{}^1=1$, $C_{12}{}^2=-\frac12$, 
$C_{22}{}^1=\phantom{A.}c$, $C_{22}{}^2=2$.
\smallbreak $\mathcal{N}_3:C_{11}{}^1=-1$, $C_{11}{}^2=0$, $C_{12}{}^1=0$, $C_{12}{}^2=-1$, $C_{22}{}^1=\phantom{}-1$, $C_{22}{}^2=0$.
\smallbreak $\mathcal{N}_4:C_{11}{}^1=-1$, $C_{11}{}^2=0$, $C_{12}{}^1=0$, $C_{12}{}^2=-1$, $C_{22}{}^1=\phantom{Aa}1$,
$C_{22}{}^2=0$.
\medbreak\noindent
\end{definition}

\begin{remark}\rm
We note that $\mathcal{N}_3$ is the Lorentzian hyperbolic plane; $\nabla$ is the Levi-Civita connection of the Lorentzian
metric $(x^1)^{-2}\{(dx^1)^2-(dx^2)^2\}$. We also note that $\mathcal{N}_4$ is the hyperbolic plane; $\nabla$ is the Levi-Civita connection
of the Riemannian metric $\{(x^1)^{-2}\{(dx^1)^2+(dx^2)^2\}$.
Since $\dim\{\mathfrak{K}(\mathcal{M})\}=3$, by Lemma~\ref{L2.4}, affine isomorphic and linear isomorphic are equivalent notions in this setting. 
\end{remark}

The following result is obtained in \cite{BGG16a}.

\begin{theorem}
Let $\mathcal{M}$ be a Type~$\mathcal{B}$ model with $\dim{\{\mathfrak{K}(\mathcal{M})\}}=3$.
Then $\mathcal{M}$ is linearly isomorphic to $\mathcal{N}^{\pm}_1$, to $\mathcal{N}^{c}_2$,  to $\mathcal{N}_3$, or to $\mathcal{N}_4$. 
Moreover, $\mathcal{N}^{\pm}_1$ and $\mathcal{N}^{c}_2$ are not of Type~$\mathcal{C}$ while 
$\mathcal{N}_3$ and $\mathcal{N}_4$ are of Type~$\mathcal{C}$.
\end{theorem}

The following result exhibits a  basis for the space of solutions of the affine quasi-Einstein equation in this setting.
It is obtained by direct computation.

\begin{theorem}\label{T4.4}
Let $\mathcal{M}$ be a Type~$\mathcal{B}$ model with $\dim{\{\mathfrak{K}(\mathcal{M})\}}=3$, then 
$E(\mu,\nabla)={0}$ except in the following cases:
\begin{enumerate}
\item $E(0,\nabla)=\operatorname{Span}\{1\}$.
\item $\mathcal{M}$ is $\mathcal{N}_{3}$, $\mu=-1$, and $E(-1,\nabla)=(x^1)^{-1}\operatorname{Span}\left\{ 1, x^2,(x^2)^2+ (x^1)^2\right\}$.
\item $\mathcal{M}$ is $\mathcal{N}_{4}$, $\mu=-1$, and $E(-1,\nabla)=(x^1)^{-1}\operatorname{Span}\left\{ 1, x^2,(x^2)^2- (x^1)^2\right\}$.
\item $\mathcal{M}$ is $\mathcal{N}_1^{\pm}$, $\mu=-\frac14$, and $E(-\frac14,\nabla)=(x^1)^{\frac{1}{2}}\operatorname{Span}\{1, x^2\}$ .
\end{enumerate}
\end{theorem}

\subsection{Homogeneous models of Type~$\mathcal{C}$}

The surfaces $\mathcal{N}_3$ and $\mathcal{N}_4$ in Definition~\ref{D4.1} correspond
 to the Riemannian and Lorentzian hyperbolic spaces. 
 These manifolds are both  Type~$\mathcal{B}$ and  Type~$\mathcal{C}$. There is one more surface of  
 Type~$\mathcal{C}$ which is not of  Type~$\mathcal{B}$, this is the affine surface defined by the Levi-Civita 
 connection of the sphere $\mathbb{S}^2$.
Since any non-flat Type~$\mathcal{C}$ surface is strongly projectively flat with Ricci tensor of rank two, Theorem \ref{T1.3} shows that 
$E(0,\nabla)=\operatorname{Span}\{ 1\}$, and  $E(\mu,\nabla)=0$ for $\mu\neq -1,0$.

Define the warped product metrics $g_1(x^1,x^2)= dx^1\otimes dx^1+  \cos^2x^1 dx^2\otimes dx^2$ and  $g_2^\pm(x^1,x^2)=\pm dx^1\otimes dx^1+  e^{2 x^1} dx^2\otimes dx^2$, which correspond to the sphere and the hyperbolic spaces in Riemannian (+) and Lorentzian(-) signatures. The Christoffel symbols associated to the corresponding Levi-Civita connections are as follows:

\begin{definition}\label{D4.5} Define the following Type~$\mathcal{C}$ structures:
$$
\begin{array}{l}
\mathcal{S}^2:\, \Gamma_{11}{}^1=0, \Gamma_{11}{}^2=0,\Gamma_{12}{}^1=0,\Gamma_{12}{}^2=-\tan x^1, \Gamma_{22}{}^1=\cos x^1\sin x^1,\Gamma_{22}{}^2=0.
\\
\mathcal{H}^2_\pm:\, \Gamma_{11}{}^1=0, \Gamma_{11}{}^2=0, \Gamma_{12}{}^1=0, \Gamma_{12}{}^2=1, \Gamma_{22}{}^1=\mp e^{2 x^1}, \Gamma_{22}{}^2=0.
\end{array}
$$
\end{definition}

\begin{remark}\rm
If $\mathcal{M}$ is non-flat and of Type~$\mathcal{C}$ then $\dim\{\mathfrak{K}(\mathcal{M})\}=3$, and $\mathcal{M}$ is 
locally affine isomorphic to $\mathcal{S}^2$ or $\mathcal{H}^2_\pm$. The models
 $\mathcal{S}^2$ and $\mathcal{H}^2_\pm$ are not affine isomorphic.
The models $\mathcal{H}^2_\pm$ are both Type~$\mathcal{C}$ and Type~$\mathcal{B}$. Moreover, they are 
locally affine isomorphic to $\mathcal{N}_3$ and $\mathcal{N}_4$, but they are not linearly isomorphic.
\end{remark}

The following theorem is a consequence of Theorem~\ref{T1.3} and a  direct computation using the three connections of Definition~\ref{D4.5}.

\begin{theorem}\label{T4.7}
Let $\mathcal{M}$ be a Type~$\mathcal{C}$ model. Then  one of the following holds:
\begin{enumerate}
\item $\mathcal{M}$ is affine isomorphic to $\mathcal{S}^2$ and\\ $E(-1,\nabla)=\operatorname{Span}\{\sin x^1, \cos x^1\cos x^2, \cos x^1 \sin x^2\}$.   
\item $\mathcal{M}$ is affine isomorphic to $\mathcal{H}^2_\pm$  and\\ $E(-1,\nabla)=\operatorname{Span}\{e^{x^1}, x^2 e^{x^1}, e^{-x^1}\pm (x^2)^2 e^{ x^1}\}$.
\end{enumerate}
Furthermore  $E(0,\nabla)=\operatorname{Span}\{1\}$ and $E(\mu,\nabla)=0$ for any $\mu\neq -1,0$.
\end{theorem}

\begin{remark}\rm
In the case of the unit sphere $\mathbb{S}^2$, the eigenfunctions corresponding to the first eigenvalue of the Laplacian are given by the restriction to $\mathbb{S}^2$ of harmonic polynomials of degree one in $\mathbb{R}^3$ (see, for example Corollary 4.49 in \cite{GHL}). Hence
$E(-1,\nabla)$ is generated by the restriction to $\mathbb{S}^2$ of the functions $\{ x,y,z\}$ given by the usual coordinates in $\mathbb{R}^3$.  Theorem \ref{T4.7}~(1) provides the local expressions of these solutions based on a warped product description of the sphere.
\end{remark}

\section{Homogeneous surfaces with $\dim\{\mathfrak{K}(\mathcal{M})\}=2$}\label{S5}

Type~$\mathcal{A}$ homogeneous surfaces with $\dim\{\mathfrak{K}(\mathcal{M})\}=2$ are determined by the fact that the Ricci tensor is of rank two (cf. Theorem 3.4 in \cite{BGG16}). Since the Ricci tensor is symmetric, it defines a flat metric on $\mathcal{M}$. This metric may be of signature $(2,0)$, $(1,1)$ or $(0,2)$, which are the cases we consider in Section 5.1. The situation in the Type~$\mathcal{B}$ setting is more subtle. First of all, the Ricci tensor is not necessarily symmetric and there are examples with skew-symmetric Ricci tensor. Indeed all but one non-flat homogeneous surfaces with $\rho_s=0$  have $\dim\{\mathfrak{K}(\mathcal{M})\}=2$ (cf. Theorem~\ref{T5.5}).
Moreover, the symmetric part of the Ricci tensor may have rank two
or rank one as in Theorem \ref{T5.10}.

\subsection{Type~$\mathcal{A}$ surfaces with $\boldsymbol{\dim\{\mathfrak{K}(\mathcal{M})\}=2}$}
We begin our study by recalling the following classification result:

\begin{theorem}{\rm\cite{BGG16a}}
Let $\mathcal{M}$ be a Type~$\mathcal{A}$ surface with $\operatorname{Rank}\{\rho\}=2$. Fix the signature $(p,q)\in\{(2,0),(1,1),(0,2)\}$ of $\rho$. Define
$\tilde\rho_{ij}:=\Gamma_{ik}{}^l\Gamma_{jl}{}^k$,
$\psi:=\operatorname{Tr}_\rho\{\tilde\rho\}=\rho^{ij}\tilde\rho_{ij}$,
$\Psi:=\det(\tilde\rho)/\det(\rho)$. Then $\psi$ and $\Psi$ are affine invariants and the following conditions are equivalent:
\begin{enumerate}
\item $\mathcal{M}_1$ is affine isomorphic to $\mathcal{M}_2$.
\item $\mathcal{M}_1$ is linearly isomorphic to $\mathcal{M}_2$.
\item $(\psi,\Psi)(\mathcal{M}_1)=(\psi,\Psi)(\mathcal{M}_2)$.
\end{enumerate}
\end{theorem}

\begin{remark}\rm
We show the image of $(\psi,\Psi)$ below; the region with red horizontal lines on the far right is the moduli space models with
positive definite Ricci tensor, the region with blue vertical lines on far left is the moduli space of models with negative definite Ricci tensor, and the central region between in white corresponds to
models with indefinite Ricci tensor.
\medbreak\begin{center}{\includegraphics[height=3.5cm,keepaspectratio=true]{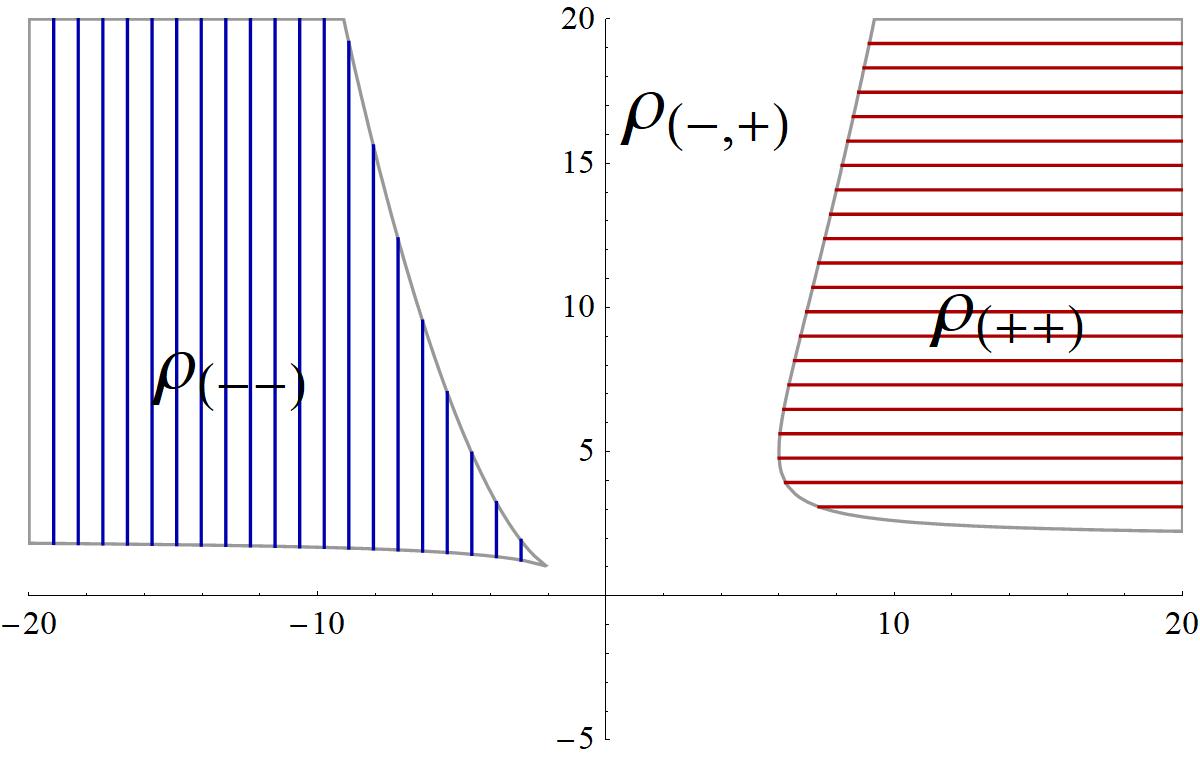}}
\captionof{figure}{Image of $(\psi,\Psi)$ for the negative definite (blue), indefinite (white) and positive definite (red) moduli spaces.}
\end{center}
\medbreak\noindent
Note that although $(\psi,\Psi)$ is 1-1 on each of the 3 cases separately, the images intersect along smooth curves.
This does not mean that $(\psi,\Psi)$ is not 1-1
on their respective domains.
\end{remark}

In what follows, $p_1(\vec x)=b_1x^1+b_2x^2$ for $(b_1,b_2)\ne(0,0)$ will denote a non-trivial linear polynomial $p_2(\vec x)=b_1x^1+b_2x^2+b_{11} (x^1)^2+2b_{12}x^1x^2+b_{22}(x^2)^2$ for $(b_{11},b_{12},b_{22})\ne(0,0,0)$ will denote a non-trivial
quadratic polynomial,  $L$ will be the line through $(7,10)$ with slope 4 and $\beta_i(\vec x)=b_{1,i}x^1+b_{2,i}x^2$ are to be distinct where
$(b_{i,1},b_{i,2})\in\mathbb{R}^2-\{(0,0)\}$. By Theorem \ref{T1.3} the only case for consideration is $\mu=-1$.
The main result of this section is the following:

\begin{theorem}
Let $\mathcal{M}$ be a Type~$\mathcal{A}$ surface model with $\dim\{\mathfrak{K}(\mathcal{M})\}=2$. Then $\dim\{E(-1,\nabla)\}=3$ and one of the following holds
\begin{enumerate}
\item If $(\psi,\Psi)\notin L$, then
 $E(-1,\nabla)=\operatorname{Span}\{e^{\vec\beta_1\cdot\vec x},e^{\vec\beta_2\cdot\vec x},e^{\vec\beta_3\cdot\vec x}\}$.
\item If $(\psi,\Psi)\in L-\{7,10\}$, then  $E(-1,\nabla)=\operatorname{Span}\{e^{\vec\beta_1\cdot\vec x},p_1(\vec x)e^{\vec\beta_1\cdot\vec x},e^{\vec\beta_2\cdot\vec x}\}$.
\item If $(\psi,\Psi)=(7,10)$, then $E(-1,\nabla)=e^{\vec\beta_1\cdot\vec x}\operatorname{Span}\{1,p_1(\vec x),p_2(\vec x)\}$.
\end{enumerate}
\end{theorem}

\begin{remark}\rm We refer to Figure~1. The line $L$ does not intersect the moduli space of negative definite
Ricci tensors. The line $L$ intersects the moduli space of indefinite Ricci tensors in the ray $(7,10)+t(1,4)$ for $t\le0$ which points down and to the
left. The line $L$ intersects the moduli space of positive definite Ricci tensors in the ray $(7,10)+t(1,4)$ for $t\ge0$ which points up and to the right.
The right boundary of the moduli space of indefinite Ricci tensors and the left boundary of the moduli space of positive definite
Ricci tensors is given by the curve $\sigma_+(t):=(4 t^2+\frac{1}{t^2}+2,4 t^4+4 t^2+2)$.
The point $(7,10)$, which corresponds to $t=1$ is an inflect point of this curve (see Figure~2). For $0<t<1$, the curve is convex to the right and for $1<t$, the
curve is convex to the left. The red curve is the boundary between the moduli space of indefinite and the moduli space of positive definite
Ricci tensors and the black line is the singular locus. We have adjusted the scale when showing the positive definite moduli space to show
the divergence of the two curves.
\medbreak\centerline{\includegraphics[height=4.75cm,keepaspectratio=true]{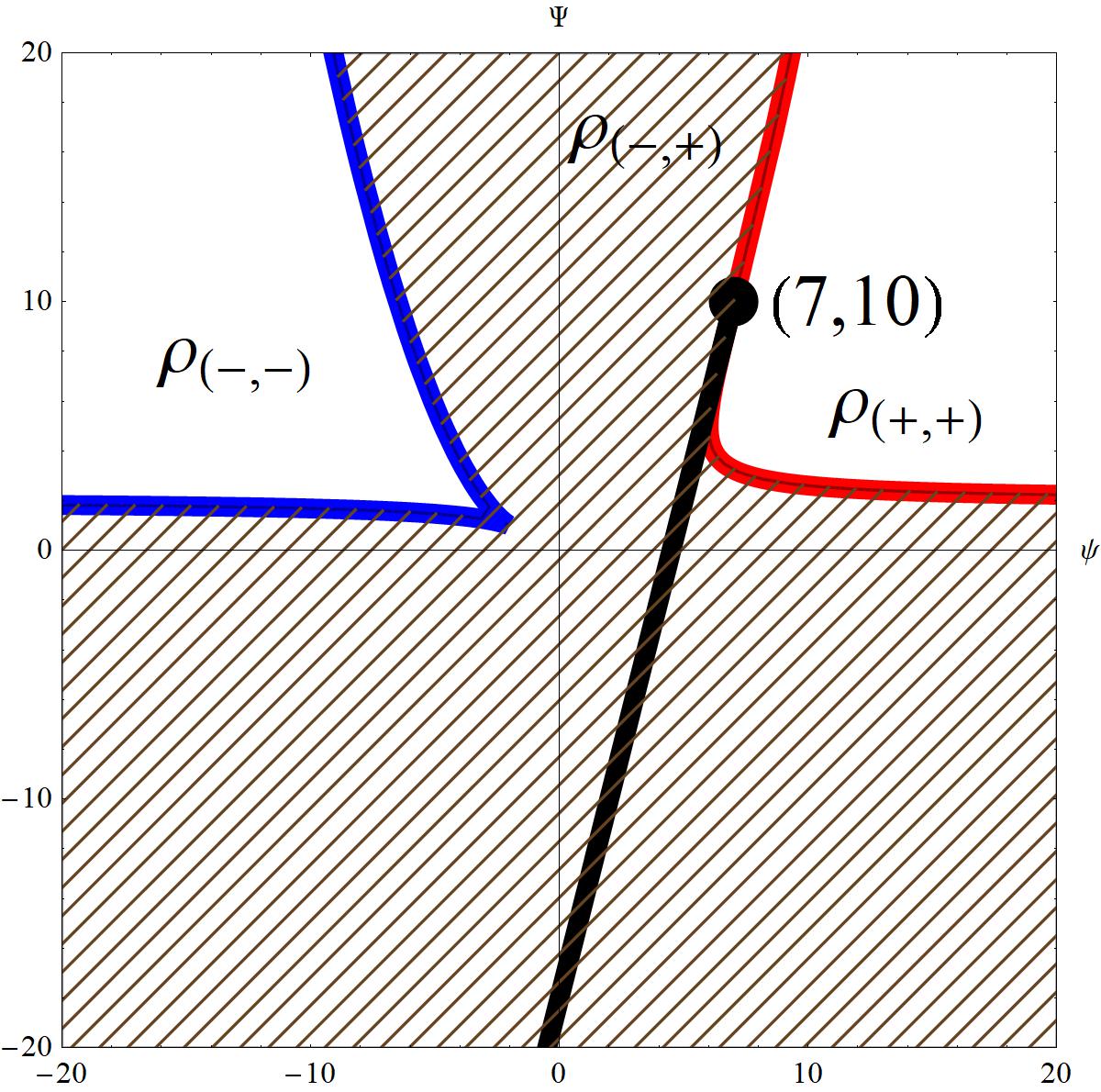}\qquad
\includegraphics[height=4.5cm,keepaspectratio=true]{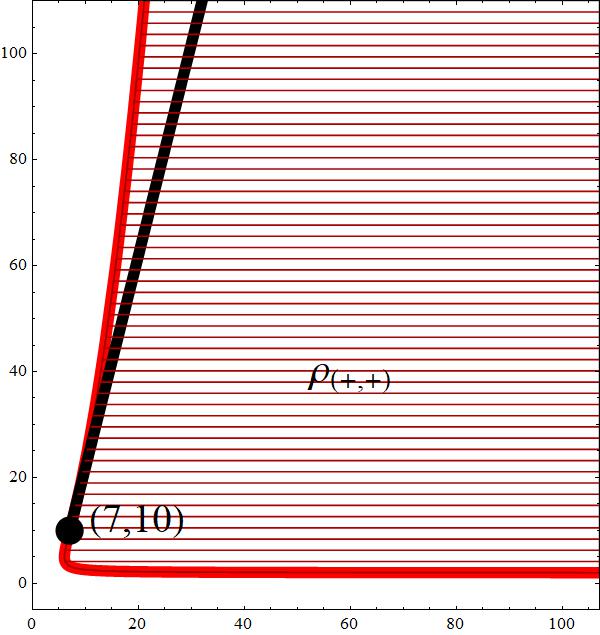}}
\captionof{figure}{ The singular locus for $\rho$ indefinite (left) and  for $\rho$ positive definite (right).}
\end{remark}

\begin{proof} We divide the proof into several parts.

\subsubsection*{Step 1. The line $L$} We show the assertions of the Theorem hold for the line $L$ as follows.
\begin{enumerate}
\item Take
$\Gamma_{11}{}^1=\Gamma_{22}{}^1=1-{c}$, $ \Gamma_{11}{}^2=-{c}$, to
obtain $(\psi,\Psi)=(7-\frac1{c},10-\frac4{c})$  and parametrize $L-(7,10)$. We compute
$\rho=\left(\begin{array}{cc}-c&c\\ c&1-c\end{array}\right)$. This is positive definite if ${c}<0$ and negative definite if ${c}>0$. We have that
$E(-1, \nabla)=\operatorname{Span}\{e^{x^1},e^{x^2},e^{x^2}(x^1+x^2(-1+\frac1{c}))\}$. Thus if $\mathcal{M}\in L-(7,10)$, the conclusion
of Assertion~(2) holds.
\item Take $\Gamma_{11}{}^1=2$, $\Gamma_{12}{}^2=\Gamma_{22}{}^1=1$, and $\Gamma_{11}{}^2=\Gamma_{12}{}^1=\Gamma_{22}{}^2=0$,
to obtain $(\psi,\Psi)=(7,10)$, $\rho=\operatorname{id}>0$, and
$E(-1,\nabla)=e^{x^1}\operatorname{Span}\{1,x^2,(2x^1+(x^2)^2)\}$. Thus if $(\psi,\Psi)=(7,10)$ and $\rho>0$, the conclusion of Assertion~(3) holds.
\item Take $\Gamma_{11}{}^1=\Gamma_{12}{}^1=\Gamma_{12}{}^2=\Gamma_{22}{}^2=-3$ and $\Gamma_{11}{}^2=\Gamma_{22}{}^1=1$ to
obtain $(\psi,\Psi)=(7,10)$, 
$E(-1,\nabla)=e^{-2(x^1+x^2)}\operatorname{Span}\{1,x^1-x^2,2x^1+(x^1-x^2)^2\}$,
and $\rho=8dx^1\otimes dx^2+8dx^2\otimes dx^1$. Thus if $(\psi,\Psi)=(7,10)$ and $\rho$
is indefinite the conclusion of Assertion~(3) holds.
\end{enumerate}

We complete the proof by showing that if $(\psi,\Psi)\notin L$, then Assertion~(1) holds. We examine the 3 moduli spaces separately;
 by Theorem \ref{T1.3} the only case for consideration is $\mu=-1$. In a particular case if $(\psi,\Psi)\in L$, we stop the analysis.

\subsubsection*{Step 2. The Ricci tensor is positive definite}
For $\mathbf{u}>0$ and $\mathbf{v}\ge0$, let
$$
\mathcal{M}_{0,2}^{\mathbf{u},\mathbf{v}}:
\Gamma_{11}{}^1=\textstyle \mathbf{u}+\frac{1}{\mathbf{u}},\ 
\Gamma_{11}{}^2=0,\ \Gamma_{12}{}^1=0,\ \Gamma_{12}{}^2=\mathbf{u},\ 
\Gamma_{22}{}^1=\mathbf{u},\ \Gamma_{22}{}^2=\mathbf{v}\,.
$$
By \cite{BGG16a}, if $\rho$ is positive definite, then $\mathcal{M}$ is linearly isomorphic to $\mathcal{M}_{0,2}^{\mathbf{u},\mathbf{v}}$.
Again, this parametrization is not 1-1. Let $f=e^{a_1x^1+a_2x^2}$.
Setting $\mathfrak{Q}_{-1} (f)=0$ yields
$$
\mathbf{u}a_1^2-a_1(1+\mathbf{u}^2)+ \mathbf{u}=0,\quad a_2(a_1-\mathbf{u})=0,\quad a_2^2-a_1\mathbf{u}-a_2\mathbf{v}+1=0\,.
$$
A direct computation shows:
\begin{enumerate}
\item If $\mathbf{u}=1$ and $\mathbf{v}=0$, one obtains $(\psi,\Psi)=(7,10)$.
\item If $\mathbf{u}=1$ and $\mathbf{v}\ne0$, then $E(-1,\nabla)=\operatorname{Span}\{e^{x^1},(x^2-x^1\mathbf{v})e^{x^1},e^{x^1+x^2\mathbf{v}}\}$. One obtains
$(\psi,\Psi)=(7,10)+\mathbf{v}^2(1,4)\in L$.
\item If $\mathbf{u}\ne1$ and $4\mathbf{u}^2+\mathbf{v}^2\ne4$, let ${\mathbf{s}}_\pm:=(\mathbf{v}\pm\sqrt{4\mathbf{u}^2+\mathbf{v}^2-4})/2$. Then\smallbreak 
$E(-1,\nabla)=\operatorname{Span}\{e^{\frac{x^1}{\mathbf{u}}},e^{\mathbf{u}x^1+{\mathbf{s}}_+x^2},e^{\mathbf{u}x^1+{\mathbf{s}}_-x^2}\}$.
\item If $\mathbf{u}\ne1$ and $4\mathbf{u}^2+\mathbf{v}^2=4$, we set $\mathbf{u}=\cos {\mathbf{t}}$ and $ \mathbf{v}=2 \sin {\mathbf{t}}$
 and obtain $(\psi,\Psi)=(7,10)+(\tan^2 \mathbf{t})(1,4)\in L$.
\end{enumerate}

\subsubsection*{Step 3. The Ricci tensor is negative definite} For $\mathbf{u}>0$ and $\mathbf{v}\ge0$, let
$$
\mathcal{M}_{2,0}^{\mathbf{u},\mathbf{v}}:\ \Gamma_{11}{}^1=\textstyle \mathbf{u}-\frac{1}{\mathbf{u}},\ \Gamma_{11}{}^2=0,\ 
\Gamma_{12}{}^1=0,\ \Gamma_{12}{}^2=\mathbf{u},\ 
\Gamma_{22}{}^1=\mathbf{u},\ \Gamma_{22}{}^2=\mathbf{v}\,.
$$
Again, this parametrization is not 1-1 but, by \cite{BGG16a}, if 
$\rho$ is negative definite, then $\mathcal{M}$ is linearly isomorphic to $\mathcal{M}_{2,0}^{\mathbf{u},\mathbf{v}}$.
Setting $\mathfrak{Q}_{-1} (f)=0$ yields
$$
 \mathbf{u}a_1^2+a_1(1-\mathbf{u}^2)- \mathbf{u}=0,\quad a_2(a_1-\mathbf{u})=0,\quad a_2^2-a_1\mathbf{u}-a_2\mathbf{v}-1=0\,.
$$
If $a_2=0$, then $a_1\mathbf{u}=-1$. 
If $a_2\ne0$, then $a_1=\mathbf{u}\ne0$ and the third equation above determines the possible values of $a_2$ to get
$$
E(-1, \nabla)=\operatorname{Span}\{e^{-\frac{x^1}{\mathbf{u}}},  e^{\mathbf{u}x^1+\frac12x^2(\mathbf{v}+\sqrt{4+4\mathbf{u}^2+\mathbf{v}^2})},e^{\mathbf{u}x^1+\frac12x^2(\mathbf{v}-\sqrt{4+4\mathbf{u}^2+\mathbf{v}^2})}\}\,.
$$

\subsubsection*{Step 4. The Ricci tensor is indefinite}

We must deal with several parametrizations in this setting.  Let
\medbreak
$\mathcal{M}_{1,1}^{\mathbf{u},\mathbf{v}}:=
\Gamma_{11}{}^1=\Gamma_{12}{}^2=\mathbf{u},\ \Gamma_{11}{}^2=
\Gamma_{22}{}^1=\sqrt{\mathbf{u}\mathbf{v}-1},\ \Gamma_{12}{}^1=\Gamma_{22}{}^2=\mathbf{v},\ \mathbf{u}\mathbf{v}>1$.
\medbreak
$\widetilde{\mathcal{M}}_{1,1}^{\mathbf{u},\mathbf{v}}:=
\Gamma_{11}{}^1=\Gamma_{12}{}^2=\mathbf{u},\ \Gamma_{11}{}^2=-\Gamma_{22}{}^1 =\sqrt{1-\mathbf{u}\mathbf{v}},\ 
\Gamma_{12}{}^1=\Gamma_{22}{}^2=\mathbf{v}$, $\mathbf{u}\mathbf{v}<1$.
\medbreak
$\widehat{\mathcal{M}}_{1,1}^{\mathbf{u},\mathbf{v},\mathbf{w}}:=
 \Gamma_{11}{}^1=\Gamma_{12}{}^2=\mathbf{u},\ \Gamma_{11}{}^2=0,\
 \Gamma_{22}{}^1=\mathbf{w},\ \Gamma_{12}{}^1=\Gamma_{22}{}^2=\mathbf{v}, \mathbf{u v} =1.
$
\medbreak\noindent The Ricci tensor then takes the form $\rho=\left(\begin{array}{ll}0&1\\1&0\end{array}\right)$. We study cases depending on the value of $\mathbf{u}\mathbf{v}$.

\medbreak\centerline{
\includegraphics[height=4.5cm,keepaspectratio=true]{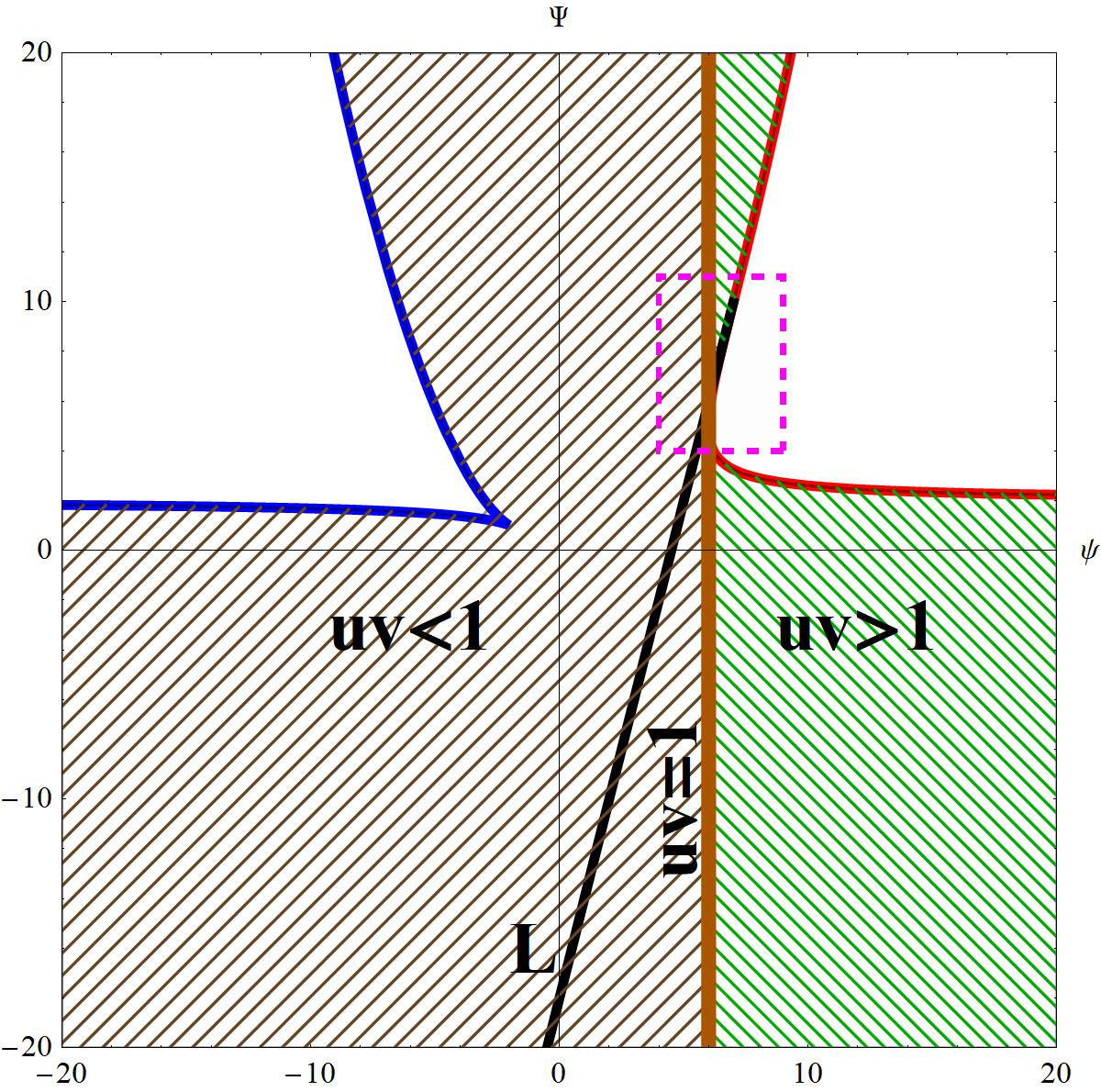}
\qquad \includegraphics[height=4.5cm,keepaspectratio=true]{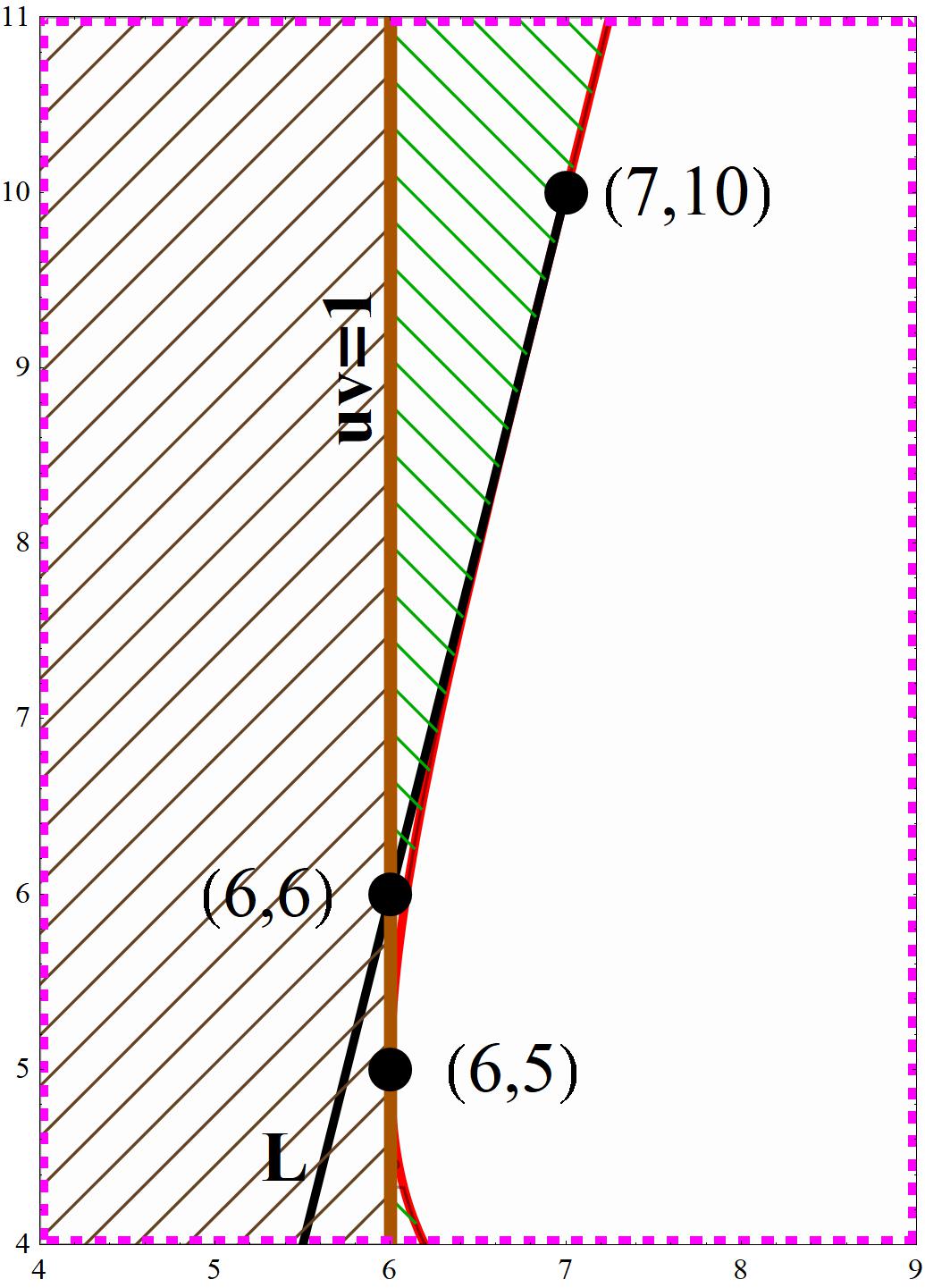}
}
\captionof{figure}{ The singular locus for $\rho$ indefinite is divided into $3$ regions corresponding to $\mathbf{u}\mathbf{v}<1$, $\mathbf{u}\mathbf{v}=1$ and $\mathbf{u}\mathbf{v}>1$.}

\subsubsection*{Case 4.1. The parametrization $\mathcal{M}_{1,1}^{\mathbf{u},\mathbf{v}}$ for $\mathbf{u}\mathbf{v}>1$}  
We set $\mathfrak{Q}_{-1} (f)=0$ to obtain an equation
$a_1^2-a_1\mathbf{u}-a_2\sqrt{\mathbf{u}\mathbf{v}-1}=0$ which determines $a_2$. Substituting this value yields a real
equation 
\begin{equation}\label{E5.a}
0=(a_1)^3-2 (a_1)^2 \mathbf{u}+a_1 \left(\mathbf{u}^2-\mathbf{v} \sqrt{\mathbf{u} \mathbf{v}-1}\right)+\sqrt{\mathbf{u} \mathbf{v}-1}\,.
\end{equation}
Generically, Equation~(\ref{E5.a}) has 3 distinct roots which will give rise to 3 different exponentials spanning $E(-1,\nabla)$ and Assertion~(1) will hold.
Equation~(\ref{E5.a}) has
a triple root when $\mathbf{u}=\mathbf{v}=-\frac3{2\sqrt2}$ and one obtains $(\psi,\Psi)=(7,10)$.
Finally, we examine when this equation has a double root.  We set $\mathbf{u}={\mathbf{t}}\sqrt{\mathbf{s}^2+1}$ and $\mathbf{v}=\frac{1}{\mathbf{t}}\sqrt{\mathbf{s}^2+1}$ for $\mathbf{s}>0$. 
This yields
$$
\Gamma_{11}{}^{1}=\Gamma_{12}{}^2={\mathbf{t}}\sqrt{1+\mathbf{s}^2},\ \Gamma_{12}{}^1=\Gamma_{22}{}^2=\frac{1}{\mathbf{t}}\sqrt{1+\mathbf{s}^2},
\Gamma_{11}{}^2=\Gamma_{22}{}^1=\mathbf{s}\,.
$$
Thereby, Equation  \eqref{E5.a} becomes
 \[a_1^3- a_1^2 2 \mathbf{t} \sqrt{\mathbf{s}^2+1} + a_1 \left(\mathbf{s}^2 t^2-\frac{\mathbf{s} \sqrt{\mathbf{s}^2+1}}{\mathbf{t}}+\mathbf{t}^2\right)+\mathbf{s}=0.\]

We suppose $c_2$ is a double root and $c_1$ is a single root. Therefore, the coefficients in $(a_1-c_1)(a_1-c_2)^2$ and in the above equation must be the same and thus we obtain the relations
$\mathbf{s}=-c_1(c_2)^2$, ${\mathbf{t}}=\frac{c_1+2c_2}{2\sqrt{1+(c_1)^2(c_2)^4}}$, and $c_1=\frac{2c_2}{1+8(c_2)^6}$. In this case,
$(\psi,\Psi)=(6,6)+\frac{32(c_2)^6}{(1+8(c_2)^6)^2}(1,4)\in L$ 
corresponds to the segment with endpoints $(6,6)$ and $(7,10)$. 

\subsubsection*{Case 4.2. The parametrization $\widetilde{\mathcal{M}}_{1,1}^{\mathbf{u},\mathbf{v}}$ for $\mathbf{u}\mathbf{v}<1$} 
We set $\mathfrak{Q}_{-1} (f)=0$ to obtain an equation
$(a_1)^2-a_1\mathbf{u}-a_2\sqrt{1-\mathbf{u}\mathbf{v}}=0$ which determines $a_2$. Substituting again we only have one more relation 
\begin{equation}\label{E5.b}
0=(a_1)^3-2(a_1)^2\mathbf{u}+a_1(\mathbf{u}^2-\mathbf{v}\sqrt{1-\mathbf{u}\mathbf{v}})+\sqrt{1-\mathbf{u}\mathbf{v}}\,.
\end{equation}
A calculation shows this equation does not have a triple root. 
We make the substitution $\mathbf{u}={\mathbf{t}}\sqrt{1-\mathbf{s}^2}$ and $\mathbf{v}=\frac{1}{\mathbf{t}}\sqrt{1-\mathbf{s}^2}$  for $\mathbf{s}>0$. Then, we use the same argument as in Case 4.2 to show that Equation~(\ref{E5.b}) has a root $c_2$ with multiplicity $2$
and a root $c_1$ with multiplicity $1$. In this case,
$\mathbf{s}=-c_1(c_2)^2$, ${\mathbf{t}}=(c_1+2c_2)/(2\sqrt{1-(c_1)^2(c_2)^4})$, and $c_1=2c_2/(1-8(c_2)^6)$. We may then compute
$(\psi,\Psi)=(6,6)-\frac{32(c_2)^6}{(1-8(c_2)^6)^2}(1,4)$ which lies on the line $L$ and corresponds to the half-line  starting at $(6,6)$.
\subsubsection*{Case 4.3. The locus $\mathbf{u}\mathbf{v}=1$}
Extending the parametrizations $\mathcal{M}_{1,1}^{\mathbf{u},\mathbf{v}}$ and $\widetilde{\mathcal{M}}_{1,1}^{\mathbf{u},\mathbf{v}}$ for $\mathbf{u}\mathbf{v}=1$ leads to $(\psi,\Psi)=(6,5)$ (see \cite{BGG16a}), which does not suffice for our purposes. Therefore we consider the slightly more general parametrization
$\widehat{\mathcal{M}}_{1,1}^{\mathbf{u},\mathbf{v},\mathbf{w}}$ for $\mathbf{u v} =1$ given by
$$
\Gamma_{11}{}^1=\Gamma_{12}{}^2=\mathbf{u},\ \Gamma_{11}{}^2=0,\
\Gamma_{22}{}^1=\mathbf{w},\ \Gamma_{12}{}^1=\Gamma_{22}{}^2=\frac{1}{\mathbf{u}}.
$$
\noindent
The image by $(\psi,\Psi)$ of this parametrization is $(6,5)- 4\mathbf{u}^3\mathbf{w}\,(0,4)$, so $\Gamma_{22}{}^1=\mathbf{w}$ parametrizes the vertical ray $x=6$ (see Figure 3).  

Setting $f=e^{a_1x^1+a_2x^2}$ and $\mathfrak{Q}_{-1} (f)=0$ yields
$$
\textstyle a_1(a_1-\mathbf{u})=0,\quad a_1(a_2-\frac{1}{\mathbf{u}})-a_2\mathbf{u}+1=0,\quad a_2(a_2-\frac{1}{\mathbf{u}})-a_1 \mathbf{w} =0\,.
$$
From the first equation either $a_1=0$ or $a_1=\mathbf{u}$. 
If $a_1=0$, then $a_2=\frac{1}{\mathbf u}$. If $a_1=\mathbf u$, then the remaining equation is $a_2^2- a_2 \frac{1}{\mathbf{u}}-\mathbf{u w} =0$. The discriminant vanishes if $\mathbf{w}=\frac{-1}{4\mathbf{u}^3}$, in 
which case $(\psi,\Psi)=(6,6)\in L$. Otherwise, there are two different solutions 
$s_\pm=\frac{1\pm \sqrt{1 + 4 \mathbf{u}^3 \mathbf{w}}}{2 \mathbf{u}}$, so
$E(-1,\nabla)=\operatorname{Span}\{ e^{x^2/\mathbf{u}}, e^{\mathbf{u}x^1+s_+ x^2}, e^{\mathbf{u}x^1+s_- x^2}\}$.\hfill\ 
\end{proof}

\subsection{Type~$\mathcal{B}$ surface models with $\boldsymbol{\dim\{\mathfrak{K}(\mathcal{M})\}=2}$}

First of all, observe that the Ricci tensor of a Type~$\mathcal{B}$ surface is not symmetric in the generic situation. Indeed, the Ricci tensor may be skew-symmetric. 
Define the following Type~$\mathcal{B}$ structures:
$$
\begin{array}{ll}
\mathcal{Q}_c:&
C_{11}{}^1=0\,,\,\, C_{11}{}^2=c\,,\,\, C_{12}{}^1=1\,,\,\, C_{12}{}^2=0\,,\,\, C_{22}{}^1=0\,,\,\, C_{22}{}^2=1, \,\,\,\, c\in\mathbb{R}.
\\ 
\noalign{\medskip}
\mathcal{P}^\pm_{0,c}:&
C_{11}{}^1=1\mp c^2,\, C_{11}{}^2=c,\, C_{12}{}^1=0,\, C_{12}{}^2=\mp c^2,\ C_{22}{}^1=\pm 1,\, C_{22}{}^2=\pm 2c,
\\ 
\noalign{\medskip}
&c>0.
\end{array}
$$ 
Then $\rho_s=0$ and thus the quasi-Einstein equation \eqref{E1.a} reduces to the Yamabe equation $\mathcal{H}f=0$, so $E(\mu,\nabla)=E(0,\nabla)$.
The following result summarizes the situation in this setting (see Lemma 4.6 and Theorem 4.12 in \cite{BGG16}).

\begin{theorem} \label{T5.5}
Let $\mathcal{M}$ be a Type~$\mathcal{B}$ surface with $\rho_s=0$. Then $\mathcal{M}$ is linearly isomorphic to a $\mathcal{Q}_c$ or a $\mathcal{P}^\pm_{0,c}$ surface. Moreover,
\begin{enumerate}
\item $\dim\{\mathfrak{K}(\mathcal{P}^+_{0,3/\sqrt{2}})\}=3$, and $\mathcal{P}^+_{0,3/\sqrt{2}}$ is linearly isomorphic to $\mathcal{N}_2^\frac 12$.
\item $\dim\{\mathfrak{K}(\mathcal{Q}_{c})\}=2$, and $\dim\{\mathfrak{K }(\mathcal{P}^\pm_{0,c})\}=2$, otherwise.
\item $E(\mu,\nabla)=E(0,\nabla)=\operatorname{Span}\{ 1\}$.
\end{enumerate}
\end{theorem}

 In Section~\ref{S5.2.1}, we discuss Yamabe solitons for generic Type~$\mathcal{B}$ surfaces with $\dim\{\mathfrak{K}(\mathcal{M})\}=2$. These
correspond to the case $\mu=0$ and are intimately related with affine gradient Ricci solitons.  In Section~\ref{S5.2.2}, we study projectively flat Type~$\mathcal{B}$ surfaces. In Section~\ref{S5.2.3},
we use these results to complete the description of the solutions to Equation~(\ref{E1.a}) for affine surfaces of Type~$\mathcal{B}$.

\subsubsection{Type~$\mathcal{B}$ Yamabe solitons}\label{S5.2.1} 
Yamabe solitons are defined by $\mathcal{H}f=0$, i.e. $\mu=0$. By Lemma~4.1 of \cite{BGG16}, $R_{ij}(df)=0$, so 
 $df$ belongs to the kernel of the curvature operator.
This provides examples of Type~$\mathcal{B}$ surfaces with $\dim\{\mathfrak{K}(\mathcal{M})\}=2$ where the Ricci tensor is of rank one.

\begin{theorem}\label{T5.6}
Let $\mathcal{M}$ be a Type~$\mathcal{B}$ surface with $\dim\{\mathfrak{K}(\mathcal{M})\}=2$. 
We have $E(0,\nabla)=\operatorname{Span}\{1\}$ except in the following cases where we
also require $\tilde\rho:=(x^1)^2\rho$ to be non-zero.
\begin{enumerate}
\item $C_{11}{}^1=-1$, $C_{12}{}^1=0$, $C_{22}{}^1=0$, $C_{22}{}^2\neq 0$, $E(0,\nabla)=\operatorname{Span}\{1,\log(x^1)\}$.
\smallbreak $\tilde\rho_{11}=-(C_{12}{}^2)^2+C_{12}{}^2+C_{11}{}^2 C_{22}{}^2$, 
$\tilde\rho_{12}=C_{22}{}^2$, $\tilde\rho_{21}=\tilde\rho_{22}=0$.
\smallbreak\item $(C_{11}{}^2,C_{12}{}^2,C_{22}{}^2)=(0,0,0)$, $E(0,\nabla)=\operatorname{Span}\{1,x^2\}$,
\smallbreak $\tilde\rho_{11}=0$, $\tilde\rho_{12}=0$, $\tilde\rho_{21}=-C_{12}{}^1$, $\tilde\rho_{22}=(C_{11}{}^1-1) C_{22}{}^1-(C_{12}{}^1)^2$.
\smallbreak\item
$(C_{11}{}^2,C_{12}{}^2,C_{22}{}^2)=-c_1(C_{11}{}^1,C_{12}{}^1,C_{22}{}^1)$, $E(0,\nabla)=\operatorname{Span}\{1,x^2+c_1x^1\}$,
\smallbreak 
$\tilde\rho_{21}= c_1 C_{11}{}^1C_{22}{}^1 -C_{12}{}^1(1+c_1 C_{12}{}^1)$, \  $\tilde\rho_{11}=c_1\tilde\rho_{21}$,
\smallbreak 
$\tilde\rho_{22}=C_{22}{}^1(C_{11}{}^1-1)-(C_{12}{}^1)^2,\ \tilde\rho_{12}=c_1\tilde\rho_{22}$.
\smallbreak\item $C_{12}{}^1=0$,
$C_{22}{}^1=0$, $C_{11}{}^1\ne-1$, $C_{22}{}^2\neq 0$, $E(0,\nabla)=\operatorname{Span}\{1,(x^1)^{C_{11}{}^1+1}\}$,
\smallbreak $\tilde\rho_{11}=-(C_{12}{}^2)^2+C_{11}{}^1 C_{12}{}^2+C_{12}{}^2
+C_{11}{}^2 C_{22}{}^2$, $\tilde\rho_{12}=C_{22}{}^2 $, $\tilde\rho_{21}=\tilde\rho_{22}=0$.
\end{enumerate}
\end{theorem}
\begin{proof} 
Let $\mathcal{M}$ be non-flat. Then  $1\leq\dim\{E(0,\nabla)\}\leq 2$ by Theorem \ref{T1.2}, since $1\in E(0,\nabla)$. We use the different possibilities in Theorem~\ref{T2.6} to investigate the existence of an additional solution of $\mathcal{H}f=0$. 
Since $\dim\{E(0,\nabla)\}\leq 2$ certain possibilities of Theorem~\ref{T2.6}~(1)  do not appear.

Let $f(x^1,x^2)=(x^1)^\alpha\{c_0+c_1\log(x^1)\}\in E(0,\nabla)$ with $c_1\neq 0$. Then a direct calculation shows
$C_{12}{}^1=0$, $C_{22}{}^1=0$, $\alpha=1+C_{11}{}^1$ and $C_{11}{}^1=-1$. Furthermore, $C_{22}{}^2\neq 0$ since otherwise $\mathcal{M}$ is both Type~$\mathcal{A}$ and Type~$\mathcal{B}$ and  $\dim\{\mathfrak{K}(\mathcal{M})\}=4$. This proves Assertion~(1).

Next we suppose $f(x^1,x^2)=(x^1)^\alpha\{x^2+c_1 x^1\}\in E(0,\nabla)$. If $\alpha\neq 0$, then $(C_{12}{}^1,C_{22}{}^1,C_{22}{}^2)=(0,0,0)$ and $\mathcal{M}$ is also of Type~$\mathcal{A}$, which is not possible. Hence 
$f(x^1,x^2)=x^2+c_1 x^1$. If $c_1=0$, then one obtains
$(C_{11}{}^2,C_{12}{}^2,C_{22}{}^2)=(0,0,0)$, which proves Assertion~(2).
If $c_1\neq 0$, then necessarily $(C_{ij}{}^2)$ and $(C_{ij}{}^1)$ are linearly dependent; this gives rise to the conditions of Assertion~(3). 

Finally, let $f(x^1,x^2)=(x^1)^\alpha\in E(0,\nabla)$ with $\alpha\neq 0$. Then 
Theorem~\ref{T2.6}~(3) shows that $C_{12}{}^1=0$, $C_{22}{}^1=0$, $C_{11}{}^1\neq -1$ and $\alpha=1+C_{11}{}^1$. 
Furthermore, $C_{22}{}^2\neq 0$ since otherwise $\mathcal{M}$ is both Type~$\mathcal{A}$ and Type~$\mathcal{B}$ and hence $\dim\{\mathfrak{K}(\mathcal{M})\}=4$. This proves Assertion~(4).
\end{proof}

\subsubsection{Strongly projectively flat Type~$\mathcal{B}$ surfaces}\label{S5.2.2}
Any strongly projectively flat surface has $\dim\{ E(-1,\nabla)\}=3$.  Previous results show that, if $\dim\{\mathfrak{K}(\mathcal{M})\}\geq 3$, then $\dim\{ E(-1,\nabla)\}\geq 1$ if and only if $\mathcal{M}$ is strongly projectively flat. While Theorem~\ref{T1.2} rules out the case  $\dim\{ E(-1,\nabla)\}=2$, we  show the existence of Type~$\mathcal{B}$ surfaces with $\dim\{ E(-1,\nabla)\}=1$ in Theorem~\ref{T5.9}.

\begin{theorem}\label{T5.7}
Let $\mathcal{M}$ be a Type~$\mathcal{B}$ model with $\dim\{\mathfrak{K}(\mathcal{M})\}=2$. Then
$\mathcal{M}$ is strongly projectively flat if and only if it is
linearly isomorphic to a surface given by
$$
C_{11}{}^1=1+2c,\, C_{11}{}^2=0,\, C_{12}{}^1=0,\, C_{12}{}^2=c,\, C_{22}{}^1=\pm 1,\, C_{22}{}^2=0,
$$ 
where $c\notin\{-1,0\}$. Then
$
\rho=(x^1)^{-2}\{c(2+c)\, dx^1\otimes dx^1\pm c\, dx^2\otimes dx^2\},
$
and moreover $E(0,\nabla)=\operatorname{Span}\{1\}$,
$E(-1,\nabla)=(x^1)^c\operatorname{Span}\{1,x^2,(x^1)^2+(x^2)^2\}$,   and $E(\mu,\nabla)=0$ otherwise.
\end{theorem}

\begin{remark}\rm
By Lemma~\ref{L7.1}, the only possible isomorphism between two surfaces in Theorem \ref{T5.7} is a shear
$T(x^1,x^2)=(x^1,ax^1+bx^2)$. Since $C_{22}{}^1=\pm1$, $b=\pm1$. Since $C_{22}{}^2=0$, we conclude $a=0$.
Since the map $(x^1,x^2)\rightarrow(x^1,-x^2)$ acts trivially on these structures, the structures of Theorem \ref{T5.7} are all inequivalent
for different values of $c\ne0$. In the cases $c=-1$, $\mathcal{M}$ corresponds to the hyperbolic plane and the Lorentzian analogue given by $\mathcal{N}_3$ and $\mathcal{N}_4$ (cf. Definition~\ref{D4.1}).
\end{remark}

\begin{proof} 

We recall that $\mathcal{M}$ is strongly projectively flat if and only if both $\rho$ and $\nabla\rho$ are totally symmetric \cite{NS}.
For a Type~$\mathcal{B}$ model one has $\rho_a=\frac{1}{2}(x^1)^{-2}(C_{12}^1 + C_{22}^2)\, dx^1\wedge dx^2$. Hence $\rho$ is symmetric if and only if  $C_{12}{}^1=-C_{22}{}^2$.
Next we examine $\nabla\rho$ and compute $(\nabla_{\partial_{x^1}} \rho)(\partial_{x^2},\partial_{x^i})-(\nabla_{\partial_{x^2}} \rho)(\partial_{x^1},\partial_{x^i})$, for $i=1,2$, to obtain the following constraints:
\begin{equation}\label{E5.c}
\begin{array}{l}
(2+C_{11}{}^1-2C_{12}{}^2)C_{22}{}^2-3C_{11}{}^2 C_{22}{}^1=0\,,\\
\noalign{\medskip}
(2-2C_{11}{}^1+4C_{12}{}^2)C_{22}{}^1+6 (C_{22}{}^2)^2=0\,.
\end{array}
\end{equation}
We  distinguish two cases. 
\subsubsection*{Case 1. $C_{22}{}^1=0$} We then obtain $C_{22}{}^2=0$. Since $C_{12}{}^1=-C_{22}{}^2$, we see 
$C_{12}{}^1=C_{22}{}^1=C_{22}{}^2=0$ and the surface is also Type~$\mathcal{A}$ by Remark~\ref{R2.1}.

\subsubsection*{Case 2. $C_{22}{}^1\ne0$} Let $\tilde x^1=x^1$ and $\tilde x^2=ax^1+x^2$ define a shear. We then obtain
$\tilde C_{12}{}^1=C_{12}{}^1-aC_{22}{}^1$. Thus by choosing $a$ appropriately, we assume that $C_{12}{}^1=0$. 
This implies $C_{22}{}^2=0$ so Equation~(\ref{E5.c}) yields  $C_{11}{}^2=0$ and $2-2C_{11}{}^1+4C_{12}{}^1=0$.
The relations in the coefficients $C_{ij}{}^k$ now follow  after rescaling the coordinates so that $C_{22}{}^1=\pm 1$. 
The converse follows
since the arguments are completely reversible.

Set $\mu=-1$ and take $f=(x^1)^\alpha$. We have
$$
\mathfrak{Q}_{\mu,22}=\pm (x^1)^{\alpha-2}(c-\alpha)\text{ and }
\mathfrak{Q}_{\mu,11}=(x^1)^{\alpha-2}(c-\alpha)(c+2-\alpha)\,.$$ Thus $\alpha=c$
and one computes
$(x^1)^c\operatorname{Span}\{1,x^2,(x^1)^2+(x^2)^2\}\subset E(-1,\nabla)$; equality now follows for dimensional reasons.

On the other hand $E(0,\nabla)=\operatorname{Span}\{1\}$  and $\dim\{ E(\mu,\nabla)\} =0$ ($\mu\neq 0,-1$) follow by Theorem \ref{T1.3} if $c\neq -2$. Setting $c=-2$ and $f=(x^1)^\alpha$, one has 
$$
\mathfrak{Q}_{\mu,11}=(x^1)^{\alpha-2}\alpha(\alpha+2)\,,
\quad
\mathfrak{Q}_{\mu,22}=-(x^1)^{\alpha-2}(\alpha-2\mu)\,,
$$
from where it follows that $\mu=-1$ and thus  $\dim\{ E(\mu,\nabla)\}=0$ otherwise.
\end{proof}

\begin{theorem}\label{T5.9}
Let $\mathcal{M}$ be a Type~$\mathcal{B}$ surface with $\dim\{\mathfrak{K}(\mathcal{M})\}=2$  which is not strongly projectively flat. Then
$\dim\{E(-1,\nabla)\}=1$ if and only if it is linearly isomorphic to one of the following surfaces
\begin{enumerate}
\item $C_{22}{}^1=0$, $C_{22}{}^2=C_{12}{}^1\neq 0$. In this case $E(-1,\nabla)=\operatorname{Span}\{(x^1)^{C_{12}{}^2} \}$. 
\item $C_{22}{}^1=\pm 1$, $C_{12}{}^1=0$,  $C_{22}{}^2=\pm 2C_{11}{}^2\neq 0$,
$C_{11}{}^1=1+2C_{12}{}^2\pm(C_{11}{}^2)^2$. In this case $E(-1,\nabla)=\operatorname{Span}\{(x^1)^{(C_{11}{}^2)^2+C_{12}{}^2} \}$. 
\end{enumerate} 
\end{theorem}

\begin{proof}
The different equivalence classes of surfaces have been obtained in \cite{BGGV17b}. Now the expression of the solutions follows by a standard calculation.
\end{proof}

\subsubsection{The general case}\label{S5.2.3} 
We work modulo linear equivalence. Assume that $\rho_s\neq 0$ and that $\mu\notin\{-1,0\}$. Let $\tilde\rho=(x^1)^2\rho$.

\begin{theorem}\label{T5.10} 
Let $\mathcal{M}$ be a Type~$\mathcal{B}$ surface with $\rho_s\ne0$. Let $\mu\in \mathbb{R}-\{0,-1\}$. Then $\dim\{E(\mu,\nabla)\}\geq 1$ if and only if (up to linear equivalence) $\mathcal{M}$ is of the form
$$
\begin{array}{l}
C_{22}{}^1=\pm1,\,\, C_{12}{}^1=0,\,\, C_{22}{}^2=\pm 2C_{11}{}^2, \,\,\,\,\Delta:=-C_{11}{}^1+C_{12}{}^2+1\ne0,
\\
\noalign{\medskip}
\mu=\Delta^{-2}\{1+2 (C_{11}{}^2)^2-(C_{11}{}^1-C_{12}{}^2)^2+2 C_{12}{}^2\}\,.
\end{array}
$$
In this case $\operatorname{Span}\{(x^1)^{\mu \Delta}\}\subset E(\mu,\nabla)$ and 
$$
\tilde\rho=\left(
\begin{array}{cc}
2 (C_{11}{}^2)^2+(C_{11}{}^1-C_{12}{}^2+1) C_{12}{}^2 & C_{11}{}^2 \\
-C_{11}{}^2 & C_{11}{}^1-C_{12}{}^2-1 \\
\end{array}
\right)\,.
$$
Moreover, $\dim\{ E(\mu,\nabla)\}=2$ if and only if one of the following holds
\begin{enumerate}
\item 
$C_{11}{}^1=\mp8 c^2-\frac{5}{2}$, $C_{11}{}^2=c $, $C_{12}{}^1=0$,
$C_{12}{}^2=\frac{1}{2} \left(\mp 8 c^2-3\right)$, 
\smallbreak\noindent$C_{22}{}^1=\pm 1$, $C_{22}{}^2=2c$, $c\ne0$.
\\
In this case $\mu=-\frac{8 c^2+3}{8c^2+4}$, and $E(\mu,\nabla)=(x^1)^\alpha\operatorname{Span}\{ 1,x^2-2c x^1\}$, where  $\alpha=-4c^2-\frac{3}{2}$. Moreover  
$\tilde\rho=\left(\begin{array}{cc} 8 \left(2 c^4\pm c^2\right) & c  \\ -c  & -4 c^2\mp 2 \\\end{array}\right)$.

\smallbreak
\item $C_{11}{}^1=c,$ $C_{11}{}^2=0$, $C_{12}{}^1=0$, $C_{12}{}^2=c+1$, $C_{22}{}^1=\pm 1$, $C_{22}{}^2=0$ with $c\notin\{-3,-1\}$.
In this case $\mu=\frac{c+1}{2}$, and $E(\mu,\nabla)=(x^1)^{1+c}\operatorname{Span}\{1,x^2\}$. Moreover $\tilde\rho=\mp2dx^2\otimes dx^2$.
\end{enumerate}
\end{theorem}

\begin{remark}\rm
None of the surfaces in Theorem~\ref{T5.10} may have  $\dim\{\mathfrak{K}(\mathcal{M})\}=4$ since $C_{22}{}^1=\pm1$ in all cases.
Moreover the affine surfaces in Theorem~\ref{T5.10}~(1) have $\dim\{\mathfrak{K}(\mathcal{M})\}=2$ since $\rho_a\neq 0$ and $\operatorname{Rank}\rho_s=2$, for all $c\neq 0$. 

The affine surfaces in Theorem~\ref{T5.10}~(2) have $\dim\{\mathfrak{K}(\mathcal{M})\}=2$ for all $c\notin\{-3,-1\}$ but the case $c=-\frac{3}{2}$. In this case $\mu=-\frac{1}{4}$ and $\mathcal{M}$ is linearly isomorphic to $\mathcal{N}_1^\pm$ in Definition \ref{D4.1} just considering the linear transformation $(u^1,u^2)=(x^1,\mp\sqrt{2}x^2)$ and Theorem~\ref{T5.10}~(2) reduces to Theorem~\ref{T4.4}~(4).
\end{remark}

\begin{proof} 
We assume first that $C_{22}{}^1=0$ and set $f=(x^1)^\alpha$ for $\alpha\neq 0$. Then 
$$
\mathfrak{Q}_{\mu,22}(f)=\mu(x^1)^{\alpha-2}C_{12}{}^1(C_{12}{}^1-C_{22}{}^2)\,.
$$
Setting $C_{12}{}^1=0$ gives $\mathfrak{Q}_{\mu,12}(f)=-\frac{1}{2}\mu C_{22}{}^2(x^1)^{\alpha-2}$. Consequently $\mathcal{M}$ is also
of Type~$\mathcal{A}$ since
$(C_{12}{}^1,C_{22}{}^1,C_{22}{}^2)=(0,0,0)$.

Assume  $C_{22}{}^1\neq 0$. A linear change of coordinates gives  $C_{22}{}^1=\pm 1$, $C_{12}{}^1=0$ (see Lemma 2.8 in \cite{BGG16}). Set $f=(x^1)^\alpha$ for $\alpha\neq 0$. Then 
$$
\begin{array}{l}
\mathfrak{Q}_{\mu,12}(f)=-\frac{1}{2}\mu (C_{22}{}^2\mp 2C_{11}{}^2)(x^1)^{\alpha-2},
\\
\noalign{\medskip}
\mathfrak{Q}_{\mu,22}(f)=\mp(x^1)^{\alpha-2}(\alpha+\mu(C_{11}{}^1-C_{12}{}^2-1)).
\end{array}
$$ 
Hence $C_{22}{}^2=\pm 2 C_{11}{}^2$ and $\alpha=\mu(1-C_{11}{}^1+C_{12}{}^2)\neq 0$. 
Now, 
$$
\begin{array}{l}
\mathfrak{Q}_{\mu,11}(f)=(x^1)^{\alpha-2}\mu\left\{ (C_{11}{}^1)^2\mp 2 (C_{11}{}^2)^2+(C_{12}{}^2)^2-2C_{12}{}^2-2C_{11}{}^1C_{12}{}^2-1\right.
\\
\noalign{\medskip}
\phantom{\mathfrak{Q}_{\mu,11}(f)=(x^1)^{\alpha-2}\mu\{ }
\left.
+(1-C_{11}{}^1+C_{12}{}^2)^2\mu
\right\}
\end{array}
$$ 
determines the admissible values of $\mu$. This proves the general case when one assumes that  $E(\mu,\nabla)\neq~\{ 0\}$.

Assume that $\dim\{ E(\mu,\nabla)\}=2$. Set $f(x^1,x^2)=(c_1x^1+c_2x^2+c_3\log(x^1))(x^1)^\alpha$. We apply Theorem~\ref{T2.6}.
The condition $\mathfrak{Q}_{\mu,22}(f)=0$ yields $c_3+x^1(c_1+2 c_2 C_{11}{}^2)=0$. Hence $c_3=0$ and $c_1=-2c_2 C_{11}{}^2$.
We now consider separately the cases $C_{11}{}^2=0$ and $C_{11}{}^2\neq 0$.

Set  $C_{11}{}^2=c\neq 0$. Then  $f(x^1,x^2)=(x^2-2c x^1)(x^1)^\alpha$, $C_{22}{}^1=\pm1$, $C_{12}{}^1=0$, and $C_{22}{}^2=\pm 2c$.
Since $\mathfrak{Q}_{\mu,11}(f)+2c\, \mathfrak{Q}_{\mu,12}(f)=0$,
$C_{11}{}^1=\frac{1}{3}(3+7 C_{12}{}^2\pm  4c^2)$. Finally the constraint 
$\mathfrak{Q}_{\mu,12}(f)=0$ gives $C_{12}{}^2=-\frac{1}{2}(3\pm 8 c^2)$. This proves Assertion~(1).

Next consider the case $C_{22}{}^1=\pm1$, $C_{12}{}^1=0$, $C_{11}{}^2=0$, $C_{22}{}^2=0$. In this case $f(x^1,x^2)=x^2(x^1)^\alpha$ and
$$
\mathfrak{Q}_{\mu,12}(f)=\frac{(C_{11}{}^1)^2+2 (C_{12}{}^2)^2-C_{12}{}^2-3 C_{11}{}^1C_{12}{}^2-1}{C_{11}{}^1-C_{12}{}^2-1}.
$$
Hence $C_{11}{}^1=C_{12}{}^2-1$ or $C_{11}{}^1=1+2C_{12}{}^2$. The latter case gives $\mu=-1$ and thus, setting $C_{11}{}^1=c$ and $C_{12}{}^2=c+1$ one has $\mu=\frac{c+1}{2}$ and Assertion (2) follows. 
\end{proof}

\section{Final Remarks}
\subsection{Flat surfaces}
An affine surface is flat if and only if $\dim\{\mathfrak{K}(\mathcal{M})\}=6$. Any flat affine surface has local coordinates $(x^1,x^2)$ with vanishing Christoffel symbols. Hence at each point $P\in M$ one has $E(P,\mu,\nabla)=\operatorname{Span}\{ 1, x^1,x^2\}$.
	Although all flat affine surfaces are locally affine equivalent, they are not necessary linearly equivalent. 

\subsubsection{Type~$\mathcal{A}$ surface models with ${\dim\{\mathfrak{K}(\mathcal{M})\}=6}$}
\begin{definition}\label{D2.1}\rm
	Define a Type~$\mathcal{A}$ model $\mathcal{M}_i=(\mathbb{R}^2,{}^{{}^\mathcal{A}\Gamma_i}\nabla)$  by
	\medbreak\qquad
	${}^\mathcal{A}\Gamma_0\equiv\Gamma_0:=\{\Gamma_{11}{}^1=0,\Gamma_{11}{}^2=0,\Gamma_{12}{}^1=0,\Gamma_{12}{}^2=0,
	\Gamma_{22}{}^1=0,\Gamma_{22}{}^2=0\}$,
	\medbreak\qquad
	${}^\mathcal{A}\Gamma_1:=\{\Gamma_{11}{}^1 = 1, \Gamma_{11}{}^2 = 0, \Gamma_{12}{}^1 = 0, \Gamma_{12}{}^2 = 1, 
	\Gamma_{22}{}^1 = 0, \Gamma_{22}{}^2 = 0\}$,
	\medbreak\qquad$
	{}^\mathcal{A}\Gamma_2:=\{\Gamma_{11}{}^1 = -1, \Gamma_{11}{}^2 = 0, \Gamma_{12}{}^1 = 0, \Gamma_{12}{}^2 = 0, 
	\Gamma_{22}{}^1 = 0, \Gamma_{22}{}^2 = 1\}$,
	\medbreak\qquad$
	{}^\mathcal{A}\Gamma_3:=\{\Gamma_{11}{}^1 = 0, \Gamma_{11}{}^2 = 0, \Gamma_{12}{}^1 = 0, \Gamma_{12}{}^2 = 0,
	\Gamma_{22}{}^1 = 0, \Gamma_{22}{}^2 = 1\}$,
	\medbreak\qquad$
	{}^\mathcal{A}\Gamma_4:=\{\Gamma_{11}{}^1 = 0, \Gamma_{11}{}^2 = 0, \Gamma_{12}{}^1 = 0, \Gamma_{12}{}^2 = 0,
	\Gamma_{22}{}^1 = 1, \Gamma_{22}{}^2 = 0\}$,
	\medbreak\qquad$
	{}^\mathcal{A}\Gamma_5:=\{\Gamma_{11}{}^1 = 1, \Gamma_{11}{}^2 = 0, \Gamma_{12}{}^1 = 0, \Gamma_{12}{}^2 = 1, 
	\Gamma_{22}{}^1 = -1, \Gamma_{22}{}^2 = 0\}$.
\end{definition}

\begin{theorem}\label{THGA} Let $\mathcal{M}=(\mathbb{R}^2,\nabla)$ be a Type~$\mathcal{A}$ model geometry with $\rho=0$. Then
	\begin{enumerate}
		\item $\nabla$
		is linearly equivalent to one of the structures ${}^{{}^\mathcal{A}\Gamma_i}\nabla$ given above. Furthermore, ${}^{{}^\mathcal{A}\Gamma_i}\nabla$ is not linearly equivalent
		to ${}^{{}^\mathcal{A}\Gamma_j}\nabla$ for $i\ne j$.
		\item $E(\mu,\nabla)$ is independent of $\mu$ and $\dim\{E(\mu,\nabla)\}=3$. Moreover:
		\begin{enumerate}
			\item $E(\mu,{}^{{}^\mathcal{A}\Gamma_0}\nabla)=\operatorname{Span}\{1,x^1,x^2\}$.
			\item $E(\mu,{}^{{}^\mathcal{A}\Gamma_1}\nabla)=\operatorname{Span}\{1,e^{x^1},x^2e^{x^1}\}$.
			\item $E(\mu,{}^{{}^\mathcal{A}\Gamma_2}\nabla)=\operatorname{Span}\{1,e^{-x^1},e^{x^2}\}$.
			\item $E(\mu,{}^{{}^\mathcal{A}\Gamma_3}\nabla)=\operatorname{Span}\{1,x^1,e^{x^2}\}$.
			\item $E(\mu,{}^{{}^\mathcal{A}\Gamma_4}\nabla)=\operatorname{Span}\{1,x^2, (x^2)^2+2x^1\}$.
			\item $E(\mu,{}^{{}^\mathcal{A}\Gamma_5}\nabla)=\operatorname{Span}\{1,e^{x^1}\cos(x^2),e^{x^1}\sin(x^2)\}$.
		\end{enumerate}
	\end{enumerate}
\end{theorem}
\begin{proof}
	Assertion (1) is proven in \cite{RBGGPV18}. Assertion (2) follows by a direct computation.
\end{proof}

The structures ${}^{{}^\mathcal{A}\Gamma_i}\nabla$ are flat. Consequently, there is a local diffeomorphism $\Phi_i$
such that $\Phi_i^*{{}^{\Gamma_0}\nabla}={}^{{}^\mathcal{A}\Gamma_i}\nabla$. 
Because $\Phi^*E(\mu,{}^{\Gamma_0}\nabla)=E(\mu,\Phi^*{}^{\Gamma_0}\nabla)$, 
we have that $\Phi_i=\Phi^*(x^i)\in E(\mu,{}^{{}^\mathcal{A}\Gamma_i}\nabla)$. Thus we can use Theorem~\ref{THGA} to determine $\Phi$ modulo a suitable affine transformation.
We have ${}^{\Gamma_0}\nabla$ is geodesically complete. Thus ${}^{\Gamma_i}\nabla$ is geodesically complete implies $\Phi$ is surjective. In \cite{DGP17}, 
D'Ascanio, Gilkey, and Pisani determined the Type~$\mathcal{A}$ surface models which are
not flat and which are geodesically complete up to linear equivalence. The following Corollary completes that analysis by examining the
flat Type~$\mathcal{A}$ surface models.

\begin{corollary} Adopt the notation established above. Then
	$$\begin{array}{lll}
	\Phi_1=(e^{x^1},x^2e^{x^1}),&\Phi_2=(e^{-x^1},e^{x^2}),&\Phi_3=(x_1,e^{x^2})\\[0.05in]
	\Phi_4=(2x^1+(x^2)^2,x^2),&\Phi_5=(e^{x^1}\cos(x^2),e^{x^1}\sin(x^2)).
	\end{array}$$
	The geometries determined by $\Gamma_1$, $\Gamma_2$, $\Gamma_3$, and $\Gamma_5$ are geodesically incomplete.
	The geometries determined by $\Gamma_0$ and $\Gamma_4$ are geodesically complete.
\end{corollary}

\subsubsection{Type~$\mathcal{B}$ surface models with ${\dim\{\mathfrak{K}(\mathcal{M})\}=6}$}

\begin{definition}\label{Def_flat_typeB_connections}\rm
		Let $0\neq c\in\mathbb{R}$. Let $\Gamma_{ij}{}^k=\frac{1}{x}C_{ij}{}^k$ define the following flat Type~$\mathcal{B}$ models
		\medbreak\qquad${}^\mathcal{B}\Gamma_0\equiv \Gamma_0:=\{{C}_{11}{}^1=0,{C}_{11}{}^2=0,{C}_{12}{}^1=0,{C}_{12}{}^2=0,
		{C}_{22}{}^1=0,{C}_{22}{}^2=0\}$,
		\medbreak\qquad${}^\mathcal{B}\Gamma_1^c:=\{{C}_{11}{}^1 = c-1, {C}_{11}{}^2 = 0, {C}_{12}{}^1 = 0, {C}_{12}{}^2 = c, 
		{C}_{22}{}^1 = 0, {C}_{22}{}^2 = 0\}$,
		\medbreak\qquad$
		{}^\mathcal{B}\Gamma_2:=\{{C}_{11}{}^1 = 1, {C}_{11}{}^2 = 0, {C}_{12}{}^1 = 0, {C}_{12}{}^2 = 0, 
		{C}_{22}{}^1 = 1, {C}_{22}{}^2 = 0\}$,
		\medbreak\qquad$
		{}^\mathcal{B}\Gamma_3^c:=\{{C}_{11}{}^1 = c, {C}_{11}{}^2 = 0, {C}_{12}{}^1 = 0, {C}_{12}{}^2 = 0,
		{C}_{22}{}^1 = 0, {C}_{22}{}^2 = 0\}$,\medbreak\qquad$
		{}^\mathcal{B}\Gamma_4:=\{{C}_{11}{}^1 = 0, {C}_{11}{}^2 = 1, {C}_{12}{}^1 = 0, {C}_{12}{}^2 = 0,
		{C}_{22}{}^1 = 0, {C}_{22}{}^2 = 0\}$,
		\medbreak\qquad$
		{}^\mathcal{B}\Gamma_5:=\{{C}_{11}{}^1 = 1, {C}_{11}{}^2 = 0, {C}_{12}{}^1 = 0, {C}_{12}{}^2 = 0, 
		{C}_{22}{}^1 = -1, {C}_{22}{}^2 = 0\}$,
		\medbreak\qquad$
		{}^\mathcal{B}\Gamma_6:=\{{C}_{11}{}^1 = -2, {C}_{11}{}^2 = 1, {C}_{12}{}^1 = 0, {C}_{12}{}^2 = -1, 
		{C}_{22}{}^1 = 0, {C}_{22}{}^2 = 0\}$.
\end{definition}

\begin{theorem}\label{THG} Let $\mathcal{M}=(\mathbb{R}^2,\nabla)$ be a Type~$\mathcal{B}$ model geometry with $\rho=0$. Then
	\begin{enumerate}
		\item $\nabla$
		is linearly equivalent to one of the structures ${}^{{}^\mathcal{B}{\Gamma_i^c}}\nabla$ given above. Furthermore, ${}^{{}^\mathcal{B}{\Gamma_i^c}}\nabla$ is not linearly equivalent
		to ${}^{{}^\mathcal{B}{\Gamma_j^{\tilde c}}}\nabla$ for $i\ne j$ or $c\neq\tilde c$.
		\item $E(\mu,\nabla)$ is independent of $\mu$ and $\dim\{E(\mu,\nabla)\}=3$. Moreover:
		\begin{enumerate}
			\item $E(\mu,{}^{{}^\mathcal{B}{\Gamma_0}}\nabla)=\operatorname{Span}\{1,x^1,x^2\}$.
			\item $E(\mu,{}^{{}^\mathcal{B}{\Gamma_1^c}}\nabla)=\operatorname{Span}\{1,(x^1)^c,(x^1)^c x^2\}$.
			\item $E(\mu,{}^{{}^\mathcal{B}{\Gamma_2}}\nabla)=\operatorname{Span}\{1,x^2,(x^1)^2+(x^2)^2\}$.
			\item $E(\mu,{}^{{}^\mathcal{B}{\Gamma_3^c}}\nabla)=\operatorname{Span}\{1,(x^1)^{c+1},x^2\}$.
			\item $E(\mu,{}^{{}^\mathcal{B}{\Gamma_4}}\nabla)=\operatorname{Span}\{1,x^1,x^1 \log(x^1)+x^2\}$.
			\item $E(\mu,{}^{{}^\mathcal{B}{\Gamma_5}}\nabla)=\operatorname{Span}\{1,x^2,(x^1)^2-(x^2)^2\}$.
			\item $E(\mu,{}^{{}^\mathcal{B}{\Gamma_6}}\nabla)=\operatorname{Span}\{1,(x^1)^{-1},(x^1)^{-1}(x^1\log(x^1)+x^2)\}$.
		\end{enumerate}
	\end{enumerate}
\end{theorem}
\begin{proof}
	Assertion (1)  was established in \cite{RBGGPV18}. Assertion (2) follows by a direct computation.
\end{proof}

\subsection{Affine gradient Ricci solitons versus affine quasi-Einstein structures}
An {\it affine gradient Ricci soliton} is a triple $(M,\nabla,f)$ where $f$ is a solution of the soliton equation $\mathcal{H} f+\rho_s=0$ on
an affine surface $\mathcal{M}=(M,\nabla)$. Homogeneous affine gradient Ricci solitons where investigated in \cite{BGG16}, showing that the non-trivial ones (i.e., solutions with non-constant function $f$) are given by one of the following:
\begin{enumerate}
\item A surface with $\dim\{\mathfrak{K}(\mathcal{M})\}=4$,
\item A Type~$\mathcal{B}$ surface $\mathcal{P}_{a,c}^\pm$, $a\neq 0$, with $\dim\{\mathfrak{K}(\mathcal{M})\}=2$ given by
$$
\begin{array}{llll}\mathcal{P}_{a,c}^\pm:&C_{11}{}^1=\frac{1}{2} \left(a^2+4 a\mp2c^2+2\right),
&C_{11}{}^2=c,&C_{12}{}^1=0,\\[0.05in]
&C_{12}{}^2=\frac{1}{2} \left(a^2+2 a\mp2c^2\right),&
C_{22}{}^1=\pm1,&C_{22}{}^2=\pm2c.
\end{array}
$$
\end{enumerate}
Surfaces with $\dim\{\mathfrak{K}(\mathcal{M})\}=4$ admit non-constant solutions of the quasi-Einstein equation for all possible eigenvalue $\mu$. In contrast, a surface $\mathcal{P}_{a,c}^\pm$ with $a\neq 0$ admits a non-constant solution of the quasi-Einstein equation if and only if $\mu=0$ and $(a,c)=(-2,0)$. In such a case $\mathcal{P}_{a,c}^\pm$ is linearly isomorphic to a surface in Theorem~\ref{T5.6}~(3) and $E(0,\nabla)=\operatorname{Span}\{1,x^2\}$.

The above shows that although the affine gradient Ricci soliton equation and the affine quasi-Einstein equation have some similarities, they behave in a quite different fashion.

\subsection{K\"ahler quasi-Einstein structures}
The existence of solutions of the quasi-Einstein equation is a rather restrictive condition in the K\"ahler setting. 

\begin{theorem}\label{T6.6}{\rm\cite{C-S-W}}
Let $(M^n, g)$ be an $n$-dimensional  Riemannian manifold with a K\"ahler quasi-Einstein metric
for $\mu\neq 0$. Then $M$ is locally a Riemannian product $M_1 \times M_2$, and $f$ can be considered as a function of $M_2$, where $M_1$ is an $(n -2)$-dimensional Einstein manifold with Einstein constant $\lambda$, and $M_2$ is a $2$-dimensional quasi-Einstein manifold.
\end{theorem}

An analogous result does not hold true in the indefinite setting and our final purpose is to present an example illustrating this situation. 
Let $(T^*M,g_{\nabla,\phi})$ be the Riemannian extension of an affine surface $\mathcal{M}=(M,\nabla)$ and let $J$ be a complex structure on $\mathcal{M}$. Take local coordinates $(x^1,x^2)$ so that $J\partial_{x^1}=\partial_{x^2}$ and consider the induced complex structure on the fibers of $T^*M$, with local coordinates $(x^1,x^2,y_1,y_2)$, determined by 
$$
\mathcal{J}\partial_{y_1}=\partial_{y_2}\,, \quad
\mathcal{J}\partial_{y_2}=-\partial_{y_1}\,.
$$
Then $\mathcal{J}$ determines an almost Hermitian structure on $(T^*M,g_{\nabla,\phi})$, which is called a \emph{proper almost Hermitian Walker structure} in \cite{Matsushita2005}, by
$$
\begin{array}{l}
\mathcal{J}\partial_{x^1}=\partial_{x^2}-g_{\nabla,\phi}(\partial_{x^1},\partial_{x^2})\partial_{y_1} +\frac{1}{2}\{ g_{\nabla,\phi}(\partial_{x^1},\partial_{x^1})-g_{\nabla,\phi}(\partial_{x^2},\partial_{x^2})\}\partial_{y_2}\,,
\\
\noalign{\medskip}
\mathcal{J}\partial_{x^2} =-\partial_{x^1}+\frac{1}{2}\{ g_{\nabla,\phi}(\partial_{x^1},\partial_{x^1})-g_{\nabla,\phi}(\partial_{x^2},\partial_{x^2})\}\partial_{y_1} +g_{\nabla,\phi}(\partial_{x^1},\partial_{x^2})\partial_{y_2}\,.
\end{array}
$$
Now, one has

\begin{lemma}\label{L6.7}
The triple $(T^*M,g_{\nabla,\phi},\mathcal{J})$ is an indefinite K\"ahler manifold if and only if
the triple $(M,\nabla,J)$ is a K\"ahler affine surface.
\end{lemma}
\begin{proof}
The conditions for a proper almost Hermitian Walker structure to be K\"ahler were considered in \cite{Matsushita2005}. Hence one has that $(T^*M,g_{\nabla,\phi},\mathcal{J})$ is K\"ahler if and only if
$$
\partial_{y_1}g_{\nabla,\phi}(\partial_{x^1},\partial_{x^1})
=-\partial_{y_1}g_{\nabla,\phi}(\partial_{x^2},\partial_{x^2})
=\partial_{y_2}g_{\nabla,\phi}(\partial_{x^1},\partial_{x^2})\,,
$$
$$
\partial_{y_2}g_{\nabla,\phi}(\partial_{x^1},\partial_{x^1})
=-\partial_{y_2}g_{\nabla,\phi}(\partial_{x^2},\partial_{x^2})
=-\partial_{y_1}g_{\nabla,\phi}(\partial_{x^1},\partial_{x^2})\,.
$$
Equivalently, the Christoffel symbols of $\mathcal{M}$ satisfy
$$
\Gamma_{11}{}^1=-\Gamma_{22}{}^1=\Gamma_{12}{}^2\text{ and }
\Gamma_{11}{}^2=-\Gamma_{22}{}^2=-\Gamma_{12}{}^1\,.
$$
and it follows from Lemma 3.6 in \cite{parallel} that these are necessary and sufficient conditions for  $(M,\nabla,J)$ to be a K\"ahler affine surface.
\end{proof}

An affine surface $\mathcal{M}$ admits a K\"ahler structure if and only if the symmetric part of the Ricci tensor $\rho_s$ is recurrent and $\operatorname{det}\{\rho_s\}>0$, i.e., $\rho_s$ is either positive definite or negative definite.
Hence, non-flat Type~$\mathcal{A}$ homogeneous surfaces do not admit any K\"ahler affine structure. Type~$\mathcal{C}$ homogeneous surfaces admit a K\"ahler affine structure provided that they correspond to the Levi-Civita connection of the sphere or the hyperbolic plane. Furthermore, in any of these cases the quasi-Einstein equation has non-trivial solutions only for $\mu=-1$ (cf. Theorem \ref{T1.3} and Theorem \ref{T4.7}). 

Finally we consider the case of Type~$\mathcal{B}$ homogeneous surfaces with $\dim\{\mathfrak{K}(\mathcal{M})\}=2$.
First of all, observe that  Type~$\mathcal{B}$ homogeneous surfaces with $\rho_s=0$ admit affine K\"ahler structures, but no non-constant solution of the affine quasi-Einstein equation. Assume therefore that $\rho_s\neq 0$.
It follows from Lemma 5.6 in \cite{parallel} that if a Type~$\mathcal{B}$ surface admits a K\"ahler structure, then
$J=\left(\begin{array}{cc}\mathfrak{t}^1{}_1&1\\
\mathfrak{t}^2{}_1&-\mathfrak{t}^1{}_1\end{array}\right)\in M_2(\mathbb{R})$,
and the Christoffel symbols satisfy
$$
\begin{array}{lll}
C_{11}{}^1=C_{22}{}^1\, \mathfrak{t}^2{}_1+2(C_{22}{}^2+2C_{22}{}^1 \, \mathfrak{t}^1{}_1)\mathfrak{t}^1{}_1,&
C_{12}{}^1=C_{22}{}^2+2C_{22}{}^1\, \mathfrak{t}^1{}_1,&
\\
\noalign{\medskip}
C_{11}{}^2=(C_{22}{}^2+2C_{22}{}^1\, \mathfrak{t}^1{}_1)\mathfrak{t}^2{}_1,& C_{12}{}^2=C_{22}{}^1\,\mathfrak{t}^2{}_1 ,
&C_{22}{}^1\neq 0.
\end{array}
$$

Set  $C_{11}{}^1=C_{12}{}^2=-C_{22}{}^1\neq 0$, and $C_{11}{}^2=-C_{22}{}^2=-C_{12}{}^1=0$. Then the previous relations show that $\mathfrak{t}^1{}_1=0$ and $\mathfrak{t}^2{}_1=-1$, thus defining a K\"ahler structure. Moreover, Theorem \ref{T5.10} shows that the affine quasi-Einstein equation admits non-constant solutions $f(x^1,x^2)=(x^1)^{1+2C_{11}{}^1}$ for $\mu=1+2C_{11}{}^1$.
Summarizing the above we have the following family of examples.

\begin{theorem}\label{T6.8}
	Let $(M,\nabla,J)$ be a K\"ahler Type~$\mathcal{B}$ affine surface with 
	$C_{11}{}^1=C_{12}{}^2=-C_{22}{}^1\neq 0$, and $C_{11}{}^2=-C_{22}{}^2=-C_{12}{}^1=0$.
Then $(T^*M,g_{\nabla,\phi},\mathcal{J})$ is an indefinite quasi-Einstein K\"ahler manifold with potential function
	$f(x^1,x^2)=(x^1)^{1+2C_{11}{}^1}$ for any $\mu=1+2C_{11}{}^1$.
	Furthermore does not split a complex line as in Theorem~\ref{T6.6}.
\end{theorem}

\end{document}